\documentclass[12pt,a4paper,reqno]{amsart}
\usepackage[left=2cm, right=2cm, bottom=3cm]{geometry} 

\usepackage[ngerman, english]{babel}

\usepackage{palatino}

\usepackage{amsmath}
\usepackage{amsthm}
\usepackage{amssymb}
\usepackage{url}
\usepackage[square,sort,comma,numbers]{natbib}
\usepackage{hyperref}
\usepackage{graphicx}
\usepackage{url}
\usepackage{listings}
\makeindex

\allowdisplaybreaks

\newtheorem{Satz}{Theorem}[section]
\newtheorem{Theorem}[Satz]{Theorem}
\newtheorem{Lemma}[Satz]{Lemma}		   
	
\newtheorem{Prop}[Satz]{Proposition}	                  
\numberwithin{equation}{section} 
\theoremstyle{definition}

\newtheorem{Definition}[Satz]{Definition} 
 
\newtheorem{Remark}[Satz]{Remark} 
\newtheorem{Assumption}[Satz]{Assumption}

\newcommand{\R}{\mathbb{R}} 
\newcommand{\Rd}{{\mathbb{R}^d}} 
\newcommand{\N}{\mathbb{N}} 

\newcommand{\MF}[1]{\mathbb{M}^p_\Filtration(\Omega,E_{[0,#1]})}
\newcommand{\MFshort}[1]{\mathbb{M}^p_\Filtration(\Omega,E_{#1})}
\newcommand{\MFshortTauEins}[1]{\mathbb{M}^p_{\Filtration^{\tau_1}}(\Omega,E_{#1})}
\newcommand{\MFHS}[1]{\mathcal{M}^p_{\Filtration^T_0}([0,T],\HS(Y,\Lzwei)}
\newcommand{\MFshortShifted}[2]{\mathbb{M}^p_{\Filtration^{#1}}(\Omega,E_{#2})}
\newcommand{\MFloc}[1]{\mathbb{M}^p_\Filtration(\Omega,E_{[0,#1)})}
\newcommand{\E}{\mathbb{E}}
\newcommand{\Prob}{\mathbb{P}}
\newcommand{\F}{\mathcal{F}}
\newcommand{\Filtration}{\mathbb{F}}

\newcommand{\HS}{\operatorname{HS}}

\newcommand{\Yosida}{R_\nu}

\newcommand{\mass}{\mathcal{M}}

\newcommand{\Lzwei}{{L^2(\Rd)}}

\newcommand{\Heins}{{H^1(\Rd)}}

\newcommand{\Hs}{{H^s(\Rd)}}

\newcommand{\LInfty}{{L^\infty(\Rd)}}

\newcommand{\LalphaPlusEins}{{L^{\alpha+1}(\Rd)}}

\newcommand{\df }{\mathrm{d}}
\newcommand{\im }{\mathrm{i}}
\newcommand{\sumM }{\sum_{m=1}^{\infty}}
\newcommand{\Real}{\operatorname{Re}}
\newcommand{\skpLzwei}[2]{\big(#1,#2\big)_{L^2}}

\newcommand{\norm}[1]{\Vert #1 \Vert}
\newcommand{\bigNorm}[1]{\left\Vert #1 \right\Vert}

\newcommand{\cutoffUdot}{\varphi_n(u)}

\newcommand{\cutoffUEinsDot}{\varphi_n(u_1)}
\newcommand{\cutoffUZweiDot}{\varphi_n(u_2)}

\newcommand{\FpSpaceIniNorm}{{\MFshort{r}}}

\newcommand{\FpSpaceIni}{{\MFshort{r}}}

\newcommand{\Addresses}{{
		\bigskip
		\footnotesize
		
		F.~Hornung, Institute for Analysis, Karlsruhe Institute for Technology (KIT), 76128 Karlsruhe, Germany\par\nopagebreak
		\textit{E-mail address}: fabian.hornung@kit.edu
		
		\medskip
		
	}}

\title{The nonlinear stochastic Schr\"odinger equation via stochastic Strichartz estimates}

\author{FABIAN HORNUNG 
	}

\date{\today}

\begin{document}

\begin{abstract}
We consider the stochastic NLS  with nonlinear Stratonovic noise for initial values in $\Lzwei$ and prove local existence and uniqueness of a mild solution for subcritical and critical nonlinearities. The proof is based on deterministic and stochastic Strichartz estimates.  In the subcritical case we prove that the solution is global, if we impose an additional assumption on the nonlinear noise. 
\end{abstract}
\maketitle



\section{Introduction}
This article studies the following stochastic nonlinear Schr\"{o}dinger equation
\begin{equation}\label{ProblemStratonovich}
\left\{
\begin{aligned}
\df u(t)&= \Big(\im \Delta u(t)-\im \lambda \vert u(t) \vert^{\alpha-1} u(t)-\frac{1}{2} \sumM \vert e_m\vert^2 \vert u(t)\vert^{2(\gamma-1)}u(t)\Big) \df t\\
&\hspace{8cm}-\im \sumM e_m \vert u(t)\vert^{\gamma-1}u(t) \df \beta_m(t),\\
u(0)&=u_0,
\end{aligned}\right.
\end{equation}	
in $\Lzwei$ with $\lambda\in \{-1,1\},$ $\alpha\in(1,1+\frac{4}{d}],$ $\gamma\in [1,1+\frac{2}{d}],$ $(e_m)_{m\in\N}\subset \LInfty$ and independent Brownian motions $\left(\beta_m\right)_{m\in\N}.$ We remark that the choice of the correction term is natural in the sense that we have 
\begin{align*}
-\im \sumM e_m \vert u(t)\vert^{\gamma-1}u(t) \circ\df \beta_m(t)=-\im \sumM e_m \vert u(t)\vert^{\gamma-1}u(t)\df \beta_m(t)-\frac{1}{2} \sumM  e_m^2 \vert u(t)\vert^{2(\gamma-1)}u(t) \df t
\end{align*}
and therefore, for conservative noise, i.e. real valued coefficients $e_m,$ \eqref{ProblemStratonovich} is a stochastic NLS with Stratonovich noise. \\
The nonlinear Schr\"odinger equation can be seen as a model for nonlinear dispersive equations and enjoys physical significance in the description of nonlinear wave phenomena. In some situations, there is a random potential in the equation which can be modeled by multiplicative Stratonovich noise. In \cite{PhysicsPaper}, the equation \eqref{ProblemStratonovich} appears with parameters $d=2,$ $\alpha=3$ and $\gamma=1,$  in the context of Scheibe aggregates with thermal fluctuations.  \\

In the literature, existence and uniqueness of the solutions to the NLS in $\Rd$ with linear multiplicative noise was studied by de Bouard and Debussche in \cite{BouardLzwei}, \cite{BouardHeins} followed by a series of papers concerning blow-up behavior and numerical studies (see \cite{deBouardDebusscheBlowUp1},\cite{deBouardDebusscheBlowUp2}, \cite{DebusscheNumericalFocusing}, \cite{deBouardDebusscheSemidiscreteScheme}) and by Barbu, R\"ockner and Zhang in \cite{BarbuL2},\cite{BarbuH1}, \cite{BarbuNoBlowUp} and \cite{ZhangStochasticStrichartz}. In \cite{StochStrichartz}, Brze\'zniak and Millet derived a new estimate for the stochastic convolution associated to the Schr\"odinger group. In contrast to \cite{BouardLzwei}, where the authors work directly with the dispersive estimate of the Schr\"odinger group, the estimate from \cite{StochStrichartz}  is based on the deterministic Strichartz inequality. This allowed them to prove global existence and uniqueness for the NLS with nonlinear Stratonovich noise on compact, two dimensional manifolds, where the dispersive estimate is not valid and has to be replaced by localized version, see \cite{Burq}. 
In this article, we show how to use the estimate from \cite{StochStrichartz} to improve the results from \cite{BouardLzwei} and \cite{BarbuL2}.\\

Let us compare our approach, assumptions and results to the articles \cite{BouardLzwei} and \cite{BarbuL2} in detail.
In \cite{BouardLzwei}, de Bouard and Debussche choose a direct approach to \eqref{ProblemStratonovich}. They reformulate the equation in a fixed point problem which they solve using Strichartz estimates complemented by an estimate of the stochastic convolution. Unfortunately, the fixed point argument only works in the case of linear noise, i.e. $\gamma=1.$ Moreover, de Bouard and Debussche impose the additional restriction $\alpha \in (1,1+\frac{2}{d-1})$ for $d\ge 3$ compared to the subcritical range $\alpha\in (1,1+\frac{4}{d})$ and their
 assumption on the noise coefficients corresponds to the square function estimate
\begin{align}\label{deBouardAssumption}
	\bigNorm{\left(\sumM \vert e_m\vert^2\right)^\frac{1}{2}}_{\Lzwei\cap L^{2+\delta}(\Rd)}<\infty
\end{align}
for some $\delta>2(d-1).$  As a result, they obtain the existence and uniqueness of a global solution $u\in L^\rho(\Omega,C([0,T],\Lzwei))\cap L^1(\Omega, L^r(0,T;L^p(\Rd)))$ for some $\rho,r,p\in (2,\infty)$ depending on $\delta$ and $\alpha,$ where $(p,r)$ is a pair of Strichartz exponents, i.e. $\frac{d}{p}+\frac{2}{r}=\frac{d}{2}.$

The approach of Barbu, R\"ockner and Zhang in \cite{BarbuL2} is different. For a finite dimensional noise $W=\sum_{m=1}^{M} e_m \beta_m,$ they reduce \eqref{ProblemStratonovich} to a non-autonomous nonlinear Schr\"odinger equation with random coefficients via the scaling transformation $u=e^{-\im W}y.$
Generally speaking, the main advantage of this approach is the fact
that the equation can be solved pathwisely. This allows to use the well known fixed point argument for the deterministic NLS (see for example \cite{Cazenave}, \cite{Linares}), as soon as Strichartz estimates or non-autonomous operators of the form
\begin{align}\label{nonautonomousOperator}
A(s):=\im \left(\Delta+b(s) \cdot \nabla+c(s)\right)
\end{align}
are available. In particular, the full range of subcritical exponents $\alpha\in (1,1+\frac{4}{d})$ is accessible and a transfer of the argument  to higher regularity is easier compared to \cite{BouardLzwei} and the present article, see Remark \ref{HeinsTransfer} below.  

However, the theory for \eqref{nonautonomousOperator} is less  developed than the theory for $\im \Delta.$
On $\Rd$, for example, one can use  results \cite{Doi96} and \cite{Marzuola}. We also refer to the recent article \cite{ZhangStochasticStrichartz} by Zhang which deals with pathwise Strichartz estimates and application of the rescaling approach to a more general class of stochastic dispersive equations on $\Rd$. However, Strichartz estimates for \eqref{nonautonomousOperator} on other geometries like compact manifolds are not available so far. Moreover, the Strichartz estimates for \eqref{nonautonomousOperator} need regular coefficients, which leads 
to the regularity and decay condition
\begin{align}\label{BarbuDecay}
e_m\in C_b^2(\Rd),\qquad \lim_{\vert \xi \vert \to \infty } \eta (\xi) \left( \vert e_m(\xi)\vert+\vert \nabla e_m(\xi)\vert+\vert \Delta e_m(\xi)\vert \right)=0 
\end{align}
with 
\begin{equation*}
\eta(\xi):=\left\{
\begin{aligned}
&1+\vert \xi \vert ^2, \hspace{4cm}d \neq 2, \\
&(1+\vert \xi \vert ^2)(\log(2+\vert \xi \vert ^2))^2,\hspace{1cm} d=2.
\end{aligned}\right.
\end{equation*}
Another drawback of the rescaling approach is the fact that the transformation $u=e^{-\im W}y$ only works for linear multiplicative noise.

Assuming  \eqref{BarbuDecay} and $\gamma=1,$ Barbu, R\"{o}ckner and Zhang prove pathwise global wellposedness of \eqref{ProblemStratonovich}, i.e. existence and uniqueness in $C([0,T],\Lzwei)\cap L^q(0,T;L^{\alpha+1}(\Rd))$ for almost all $\omega\in \Omega$ and almost sure continuous dependence on the data,  for $\alpha\in (1,1+\frac{4}{d})$ and pathwise local wellposedness for $\alpha=1+\frac{4}{d},$  see \cite{BarbuL2}, Theorem 2.2 and Corollary 5.2. \\

The present article is motivated by the following two goals: 
\begin{itemize}
	\item We would like to treat nonlinear noise and weaken the assumption \eqref{BarbuDecay} from \cite{BarbuL2}.
	\item We would like to avoid the restriction from \cite{BouardLzwei} and allow the full range of subcritical exponents $\alpha\in (1,1+\frac{4}{d})$ in a global existence and uniqueness result.
\end{itemize}
To archieve this, we apply the direct approach by de Bouard and Debussche, but we substitute their estimate of the stochastic convolution by the stochastic Strichartz estimate due to Brzezniak and Millet, which reads
\begin{align}\label{StrichartzBrzezniak}
\E \bigNorm{\int_0^\cdot U(\cdot-s)\varPhi(s)\df \beta(s)}_{L^q(0,T;L^p)}^r
\lesssim \E \norm{\varPhi}_{L^2(0,T;L^2)}^r
\end{align}
in a simplified form (see Proposition \ref{StochStrichartz} for the details). Here, $(q,p)$ is an arbitrary Strichartz pair and thus, the restriction of $\alpha$ can be avoided. Moreover,  we observe that  the stochastic convolution improves  integrability in time and space from $2$ to $q>2$ and $p>2,$ respectively. This can be used to deal with nonlinear noise. The results of this article are compressed in the following Theorem.

\begin{Theorem}\label{mainTheorem}
	Let $u_0\in\Lzwei,$ $\lambda\in\{-1,1\},$  $\left(\beta_m\right)_{m\in\N}$ be a sequence of independent Brownian motions and $(e_m)_{m\in\N}\subset \LInfty$ with 
	\begin{align}\label{noiseBoundsHIntroduction}
	\sumM \norm{e_m}_{L^\infty}^2<\infty.
	\end{align}
	Then, the following assertions hold:
	\begin{enumerate}
		\item[a)] Let $\alpha\in (1,1+\frac{4}{d}]$ and $\gamma \in [1,1+\frac{2}{d}].$ Then, there is a unique local mild solution of \eqref{ProblemStratonovich} in $\Lzwei.$
		\item[b)] Let $\alpha\in (1,1+\frac{4}{d})$ and $\gamma=1.$	
		Then,  the solution from a) is global.
				\item[c)] Let $e_m$ be real valued for each $m\in \N,$ $\alpha\in (1,1+\frac{4}{d})$ and 
					\begin{align*}
					1< \gamma <\frac{\alpha-1}{\alpha+1} \frac{4+d(1-\alpha)}{4\alpha+d(1-\alpha)}+1.
					\end{align*}
				Then,  the solution from a) is global.
	\end{enumerate}
\end{Theorem}

Let us briefly sketch the content the following sections devoted to the proof of Theorem \ref{mainTheorem}. In the next section, we fix the notations and assumptions, introduce the solution concept and recall the deterministic and stochastic Strichartz estimates. In the third paragragh, we prove local existence and uniqueness of \eqref{ProblemStratonovich} and treat the cases $\gamma=1$ and $\gamma\neq 1$ at once. To this end, we solve the slightly more general truncated problem
 \begin{equation}\label{ProblemStratonovichApproxIntro}
 \left\{
 \begin{aligned}
 \df u_n(t)=& \left(\im \Delta u_n(t)-\im \varphi_n(u_n,t) \vert u_n(t) \vert^{\alpha-1} u_n(t)\right) \df t\\&
 -\frac{1}{2}\sumM \left[\vert e_m\vert^2 \varphi_n(u_n,t)\vert u_n(t) \vert^{2(\gamma-1)} u_n(t)+B_m^* B_m u_n(t)\right]\df t\\&
 -\im\sumM \left[ \varphi_n(u_n,t) e_m   \vert u_n(t) \vert^{\gamma-1} u_n(t)  +B_m u_n(t)\right]\df \beta_m(t),\\
 u(0)=&u_0,
 \end{aligned}\right.
 \end{equation}		
where the noise term and the Stratonovich term are split into a nonlinear part and a linear one with general bounded operators $B_m$ on $\Lzwei.$ 
Moreover, we introduce a truncation function $\varphi_n(u_n,t)=\theta_n(Z_t(u_n))$ for a process
\begin{align}\label{definitionZ}
	Z_t(u_n):=\norm{u_n}_{L^q(0,t;L^{\alpha+1})}+\norm{u_n}_{L^{\tilde{q}}(0,t;L^{2\gamma})}
\end{align}
and 
 $\theta_n:[0,\infty)\to [0,1]$ with $\theta_n(x)=1$ for $x\in [0,n]$ and $\theta_n(x)=0$ for $x\ge 2n.$ In \eqref{definitionZ}, $q$ and $\tilde{q}$ are chosen such that $(\alpha+1,q)$ and $(2\gamma,\tilde{q})$ are Strichartz pairs.
 In order to construct a solution of \eqref{ProblemStratonovichApproxIntro}, we use a fixed point argument in the natural space $L^q(\Omega,E_T),$ where 
  \begin{align*}
  E_T:=\left\{
  \begin{aligned}
  &L^q(0,T;\LalphaPlusEins)\cap C([0,T],\Lzwei), \hspace{1.5cm} \alpha+1\ge 2\gamma, \\
  &L^{\tilde{q}}(0,T;L^{2\gamma}(\Rd))\cap C([0,T],\Lzwei),\hspace{1.75cm} \alpha+1< 2\gamma,
  \end{aligned}\right.
  \end{align*}
  depends on the dominant nonlinearity in \eqref{ProblemStratonovichApproxIntro}.  The truncation replaces the restriction to balls in $E_T,$ which is used in the deterministic setting, and permits the pathwise application of the deterministic Strichartz estimates.
Since the solution of \eqref{ProblemStratonovichApproxIntro} also solves the untruncated problem up to the stopping time 
\begin{align}\label{existenceTimes}
\tau_n:=\inf \left\{t\ge 0: Z_t(u_n)\ge n\right\}\land T,
\end{align} 
this yields a local solution $u$  to \eqref{ProblemStratonovich}  up to time $\tau_\infty:=\sup_{n\in\N}\tau_n.$  The uniqueness of the solution to \eqref{ProblemStratonovich} can be reduced to the uniqueness of \eqref{ProblemStratonovichApproxIntro}.\\
In the critical setting, i.e. $\alpha=1+\frac{4}{d}$ or $\gamma=1+\frac{2}{d},$ a similar argument yields a local solution. Note that in this case, we use the truncation $\varphi_\nu$ for a small $\nu\in(0,1)$ instead of $\varphi_n$ for a large $n\in\N.$ 

The definition of the existence times in \eqref{existenceTimes} shows that uniform bounds on the norms in \eqref{definitionZ} imply global existence. This is not convenient, since one would prefer a blow-up criterium including
the $L^2$-norm, which can be controlled by the Hamiltonian structure of the NLS. For linear noise, de Bouard and Debussche developed a strategy to get 
\begin{align*}	
\sup_{n\in\N} \E \Big[\norm{u_n}_{L^q(0,T;L^{\alpha+1})}\Big]\le C_T,
\end{align*}
see \cite{BouardLzwei}, Proposition 4.1. In the fourth section, we adapt this proof to our situation in order to prove part b) and c) of Theorem \ref{mainTheorem}.

\section{Setting and Strichartz estimates}

In this section, we introduce some notations, assumptions and solution concepts
and recall deterministic and stochastic Strichartz estimates, which will be used to construct the local solution.

\begin{Assumption}\label{stochasticAssumptions}
	We assume the following: 
	\begin{itemize}
		\item[i)] We fix  $d\in\N$ and $T>0.$ Moreover, let $u_0\in\Lzwei,$ $\lambda\in\{-1,1\}.$ We denote the Schr\"odinger group, i.e. the $C_0$-group of unitary operators generated by $\im \Delta,$  by $\left(U(t)\right)_{t\in\R}.$ 
		\item[ii)] Let $(\Omega,\F,\Prob)$ be a probability space, $Y$ be a separable Hilbert space with ONB $(f_m)_{m\in\N}$ and $W$ a cylindrical Wiener process in $Y$ adapted to a filtration $\Filtration$ satisfying the usual conditions.
		\item[iii)] 
		Let $\left(e_m\right)_{m\in\N}\subset \LInfty$ and  $\left(B_m\right)_{m\in\N}\subset \mathcal{L}(\Lzwei)$ with
		\begin{align*}
			\sumM \norm{e_m}_{L^\infty}^2<\infty,\qquad \sumM \norm{B_m}_{\mathcal{L}(L^2)}^2<\infty
		\end{align*}
		and define the linear bounded operators
		$B_1,B_2: \Lzwei \to \HS(Y,\Lzwei)$ by 
		\begin{align*}
			B_1(u)f_m:=e_m u,\qquad B_2(u)f_m:=B_m u,\qquad u\in\Lzwei,\quad m\in\N.
		\end{align*}
	\end{itemize}
	
\end{Assumption}
For presentation purposes, we used in the introduction that the process
\begin{align*}
	W=\sumM f_m \beta_m
\end{align*}
with a sequence $\left(\beta_m\right)_{m\in\N}$ of independent Brownian motions is a cylindrical Wiener process in $Y,$ see \cite{daPrato}, Proposition 4.7. 
In the proof of Theorem \ref{mainTheorem}, we want to avoid to treat the cases of linear noise, i.e. $\gamma=1,$ and nonlinear noise, i.e. $\gamma\neq 1,$ separately. Therefore, we study the slight generalization of \eqref{ProblemStratonovich} given by
\begin{equation}\label{ProblemStratonovichBasics}
\left\{
\begin{aligned}
\df u(t)&= \left[\im \Delta u(t)-\im \lambda \vert u(t) \vert^{\alpha-1} u(t)+\mu_1\left(\vert u(t)\vert^{2(\gamma-1)}u(t)\right)+\mu_2(u(t))\right] dt\\
&\hspace{8cm}-\im \left[B_1\left(\vert u(t)\vert^{\gamma-1}u(t)\right)+B_2 u(t)\right] \df W(t),\\
u(0)&=u_0,
\end{aligned}\right.
\end{equation}
where 
\begin{align*}
	\mu_1:=-\frac{1}{2}\sumM \vert e_m\vert^2,\qquad \mu_2:=-\frac{1}{2}\sumM B_m^*B_m.
\end{align*}
Since we look for mild solutions of \eqref{ProblemStratonovichBasics}, we reformulate the equation in the form
		\begin{align}\label{mildEquation1}
		u(t)=&  U(t)u_0+ \int_0^{t}  U(t-s)\left[-\im \lambda\vert u(s)\vert^{\alpha-1} u(s)+\mu_1\left(\vert u(s)\vert^{2(\gamma-1)}u(s)\right)+\mu_2(u(s))\right] \df s\nonumber\\&\hspace{2cm}- \im \int_0^{t} 
		U(t-s)\left[B_1\left(\vert u(s)\vert^{\gamma-1}u(s)\right)+B_2 u(s)\right] \df W(s)
		\end{align}
In the following two Propositions, we introduce the main tool to apply a fixed argument to solve \eqref{mildEquation1}, namely the Strichartz estimates. 

\begin{Prop}[Deterministic Strichartz Estimates]\label{DetStrichartz}
	Let $r_j,q_j\in [2,\infty],$ $j=1,2,$ with
	\begin{align*}
	\frac{2}{q_j}+\frac{d}{r_j}=\frac{d}{2},\qquad (q_j,r_j,d)\neq(2,\infty,2).
	\end{align*} 
	Let $x\in \Lzwei,$ $J\subset \R$ an interval with $0\in J$ and $f\in L^{q_2'}(J,L^{r_2'}(\Rd)).$ Then, there is a constant $C>0$ independent of $J,f$ and $x$ such that
	\begin{enumerate}
		\item[a)] $\Vert U(\cdot) x\Vert_{L^{q_1}(J,L^{r_1})} \le C \Vert x\Vert_{L^2},$
		\item[b)] $ \Vert\int_0^\cdot U(\cdot-s)f(s)\df s\Vert_{L^{q_1}(J,L^{r_1})}\le C \Vert f\Vert_{L^{q_2'}(J,L^{r_2'})}.$
	\end{enumerate}
	Furthermore, $U(\cdot) x$ and $\int_0^\cdot U(\cdot-s)f(s)\df s$ are elements of $C_b(J,{L^2}(\Rd))$ and we have
	\begin{enumerate}
		\item[c)] $\Vert U(\cdot) x\Vert_{C_b(J,{L^2})} \le C \Vert x\Vert_{L^2},$
		\item[d)] $ \Vert\int_0^\cdot U(\cdot-s)f(s)\df s\Vert_{C_b(J,{L^2})}\le C \Vert f\Vert_{L^{q_2'}(J,L^{r_2'})}.$
	\end{enumerate}
\end{Prop}	

\begin{proof}
	These estimates are well known, see for example \cite{Cazenave}, Theorem 2.3.3.
\end{proof}

The estimates from Proposition \ref{DetStrichartz} can be used to deal with the free evolution and the deterministic convolution in \eqref{mildEquation1}. Furthermore, we need an estimate of the stochastic convolution. 
In order to apply Banach's fixed point Theorem iteratively, we have to deal with initial times $T_0\ge 0.$
We denote the shifted filtration $\left(\F_{t+T_0}\right)_{t\ge 0}$ by $\Filtration^{T_0}.$ 
The process given by 
\begin{align*}
W^{T_0}(t):=W(T_0+t)-W(T_0),\qquad t\ge 0,
\end{align*}
is a cylindrical Wiener process with respect to $\Filtration^{T_0}.$
For $T_1>0$ and a $\Filtration^{T_0}$-predictable process $\varPhi\in L^r(\Omega, L^2(0,T_1;\HS(Y,\Lzwei))),$ we define 
\begin{align}\label{DefStochasticConvolution}
J_{[0,T_1]}^{T_0} \varPhi (t):=\int_{0}^{t} U(t-s) \varPhi(s) \df W^{T_0}(s),\qquad t\in [0,T_1],
\end{align}
by the stochastic integration theory in the Hilbert space $\Lzwei,$ see \cite{daPrato}, chapter 4. 
 Note that for an $\Filtration$-predictable process $\varPhi,$ we have
 \begin{align}\label{shiftStochIntegral}
 \int_0^t U(t-s)\varPhi(T_0+s)\df W^{T_0}(s)=\int_{T_0}^{T_0+t}U((T_0+t)-s)\varPhi(s)\df W(s)
 \end{align}
 almost surely for all $t.$
Since we are also interested in Strichartz estimates, we need a definition of the right hand side of \eqref{DefStochasticConvolution} in $L^q(0,T;L^p(\Rd))$-spaces for $q,p>2.$
This can be done by the theory of stochastic integration in martingale type 2 spaces, see \cite{BrzezniakConvolutions} and the references therein or in UMD spaces, see \cite{UMDStochIntegration}.
The tool to estimate the stochastic convolution \eqref{DefStochasticConvolution} is  the following result due to Brze\'zniak and Millet, \cite{StochStrichartz}.
\begin{Prop}[Stochastic Strichartz Estimates]\label{StochStrichartz}
	Let $T_1>0,$ $p\in(1,\infty)$ and $q,r\in [2,\infty]$  with
	\begin{align*}
	\frac{2}{q}+\frac{d}{r}=\frac{d}{2},\qquad (q,r,d)\neq(2,\infty,2).
	\end{align*}
	For all $\Filtration^{T_0}$-predictable processes $\varPhi\in L^p(\Omega, L^2(0,T_1;\HS(Y,\Lzwei))),$  $J_{[0,T_1]}^{T_0}\varPhi$ is continuous in $\Lzwei$ and $\Filtration^{T_0}$-adapted
	with
	\begin{align*}
	\norm{J_{[0,T_1]}^{T_0}\varPhi}_{L^p(\Omega, L^q(0,T_1,L^r)}\lesssim \norm{\varPhi}_{L^p(\Omega, L^2(0,T_1;\HS(Y,L^2)))}
	\end{align*}
	and
	\begin{align*}
	\norm{J_{[0,T_1]}^{T_0}\varPhi}_{L^p(\Omega, C([0,T_1],L^2)}\lesssim \norm{\varPhi}_{L^p(\Omega, L^2(0,T_1;\HS(Y,L^2)))}.
	\end{align*}
\end{Prop}
\begin{proof}
	See \cite{StochStrichartz}, Theorem 3.10, Proposition 3.12 and Corollary 3.13 for the statement in the case $q=p.$ The proof is based on the Burkholder-Davis-Gundy inequality 
	\begin{align*}
		\E \sup_{t\in [0,T_1]} \bigNorm{\int_0^t \varPhi(s) \df W^{T_0}(s)}_X^p
		\lesssim \E \left(\int_0^{T_1}\norm{\varphi(s)}_X^2 \df s\right)^{\frac{p}{2}}
	\end{align*}
	for $X=\Lzwei$ and $X=L^q(0,T_1;L^r(\Rd)),$ which holds for arbitrary $p\in(1,\infty).$ Therefore, $q=p$ is not needed. 
	For the BDG-inequality in martingale type 2 spaces we refer to \cite{BrzezniakConvolutions}, Theorem 2.4.  
\end{proof}
Next, we introduce the Banach spaces for the fixed point argument depending on the powers $\alpha$ and $\gamma.$  
For $\alpha\in \left(1,1+\frac{4}{d}\right]$ and $\gamma \in \left(1,1+\frac{2}{d}\right],$ we fix  $q,\tilde{q}\in (2,\infty)$ such that 
\begin{align}\label{StrichartzAlphaQ}
\frac{2}{q}+\frac{d}{\alpha+1}=\frac{d}{2},\qquad 
\frac{2}{\tilde{q}}+\frac{d}{2\gamma}=\frac{d}{2} 
\end{align} 
in order to apply the Strichartz estimates from Propositions \ref{DetStrichartz} and \ref{StochStrichartz} with the exponent pairs $(\alpha+1,q)$ and $(2\gamma,\tilde{q} ).$ Moreover, we set 
\begin{align*}
	Y_{[a,b]}:=\left\{
	\begin{aligned}
	&L^q(a,b;\LalphaPlusEins), \hspace{1.5cm} \alpha+1\ge 2\gamma, \\
	&L^{\tilde{q}}(a,b;L^{2\gamma}(\Rd)),\hspace{1.75cm} \alpha+1< 2\gamma,
	\end{aligned}\right.
\end{align*}
and
\begin{align*}
 E_{[a,b]}:=Y_{[a,b]} \cap C([a,b],\Lzwei),\qquad 0\le a\le b\le T.
\end{align*}
We remark that in the critical case with $\alpha=1+\frac{4}{d}$ or $\gamma=1+\frac{2}{d},$ the Strichartz exponents for time and space coincide and we get $Y_{[a,b]}=L^{2+\frac{4}{d}}(a,b;L^{2+\frac{4}{d}}(\Rd)).$
The relationship between the spaces from above is clarified by the following interpolation Lemma.
%
\begin{Lemma}\label{interpolation}
	We have
	\begin{align*}
	E_{[a,b]}\hookrightarrow L^{q}(a,b;L^{\alpha+1}(\Rd)) \cap L^{\tilde{q}}(a,b;L^{2\gamma}(\Rd)).
	\end{align*}
\end{Lemma}

\begin{proof}
	We treat $\alpha+1\ge 2\gamma.$ The other case can be proved analogously. Since $\gamma\in (1,\frac{\alpha+1}{2}],$ we can take $\theta\in (0,1]$ with
	\begin{align*}
	\frac{1}{2\gamma}=\frac{\theta}{\alpha+1}+\frac{1-\theta}{2}.
	\end{align*}
	By the scaling conditions \eqref{StrichartzAlphaQ}, we also get $\frac{1}{\tilde{q}}=\frac{\theta}{q}.$ Hence 
	\begin{align*}
	\norm{u}_{L^{\tilde{q}}(a,b;L^{2\gamma})}^{\tilde{q}}
	&\le \int_a^b \norm{u(s)}_{L^{\alpha+1}}^{\tilde{q}\theta}\norm{u(s)}_{L^2}^{\tilde{q}(1-\theta)}\df s\nonumber\\
	&\le \norm{u}_{L^\infty(a,b;L^2)}^{\tilde{q}(1-\theta)}\int_a^b \norm{u(s)}_{L^{\alpha+1}}^q \df s
	=\norm{u}_{L^\infty(a,b;L^2)}^{\tilde{q}(1-\theta)}\norm{u}_{L^q(a,b;L^{\alpha+1})}^{\tilde{q}\theta} 
	\le \norm{u}_{E_{[a,b]}}^{\tilde{q}}
	\end{align*}
	for $u\in E_{[a,b]}$ by Lyapunov's inequality and we have
	\begin{align*}
		\norm{u}_{L^{q}(a,b;L^{\alpha+1})}+\norm{u}_{L^{\tilde{q}}(a,b;L^{2\gamma})}\le 2 \norm{u}_{E_{[a,b]}},\qquad u\in E_{[a,b]}.
	\end{align*}
\end{proof}
Furthermore, we abbreviate $Y_r:=Y_{[0,r]}$ and $E_r:=E_{[0,r]}$ for $r>0.$ 
Let $\tau$ be an $\Filtration$-stopping time and $p\in(1,\infty).$ Then, we denote by $\MF{\tau}$ the  space of  processes $u:[0,T]\times \Omega\to \Lzwei\cap L^{2\gamma}(\Rd)$ with continuous paths in $\Lzwei$ which are $\Filtration$-adapted in $\Lzwei$ and $\Filtration$-predictable in $L^{2\gamma}(\Rd)$ such that
\begin{align*}
	\norm{u}_{\MF{\tau}}^p:=\E \left[\sup_{t\in [0,\tau]}\norm{u(t)}_{L^2}^p+\norm{u}_{Y_\tau}^p\right]<\infty.
\end{align*}
Often, we abbreviate $u\in \MFshort{\tau}:=\MF{\tau}.$
Moreover, we say $u\in\MFloc{\tau}$ if $u$ is a continuous $\Filtration$-adapted process in $\Lzwei$  and there is a sequence $\left(\tau_n\right)_{n\in\N}$ of stopping times with $\tau_n \nearrow \tau$ almost surely as $n\to \infty,$ such that $u\in \MF{\tau_n}$ for all $n\in\N.$ 
\begin{Definition}\label{SolutionDef}
	Let   $\alpha\in (1,1+\frac{4}{d}],$  $\gamma\in (1,1+\frac{2}{d}]$ and $p\in (1,\infty).$
	\begin{enumerate}
		\item[a)]
		A \emph{local mild solution} of \eqref{ProblemStratonovichBasics} is a  triple $\left(u,\left(\tau_n\right)_{n\in\N},\tau\right)$ consisting of stopping times $\tau,\tau_n,$ $n\in\N,$ with $\tau_n\nearrow\tau$ almost surely as $n\to \infty,$ and a  process $u\in \MFloc{\tau},$
		such that $u\in \MF{\tau_n}$ and 
		\begin{align}\label{mildEquation}
		u(t)=&  U(t)u_0+ \int_0^{t}  U(t-s)\left[-\im \lambda\vert u(s)\vert^{\alpha-1} u(s)+\mu_1\left(\vert u(s)\vert^{2(\gamma-1)}u(s)\right)+\mu_2(u(s))\right] \df s\nonumber\\&\hspace{2cm}- \im \int_0^{t} 
		U(t-s)\left[B_1\left(\vert u(s)\vert^{\gamma-1}u(s)\right)+B_2 u(s)\right] \df W(s)
		\end{align}
		almost surely on $\left\{t\le \tau_n\right\}$ in $\Lzwei$ for all $n\in\N.$ Often, we shortly denote the local mild solution by $\left(u,\tau\right).$
		\item[b)] Solutions of \eqref{ProblemStratonovich} are called \emph{unique}, if we have 
		\begin{align*}
		\Prob\Big(u_1(t)=u_2(t)\quad \forall t\in [0,\sigma_1\land \sigma_2)\Big)=1
		\end{align*}
		for all local mild solutions $(u_1,\sigma_1)$ and $(u_2,\sigma_2).$
		\item[c)] 
		A local mild solution $\left(u,\tau\right)$ with $\tau=T$ almost surely and $u\in \MF{T}$ is called \emph{global mild solution}.  
	\end{enumerate}
\end{Definition}

\section{Truncated equation, local existence and uniqueness}

This section is devoted to the proof of the local part of Theorem $\ref{mainTheorem}.$ 
 In order to transfer the deterministic fixed point argument, see \cite{Linares}, Theorems 5.2 and 5.3,  to the stochastic setting, we would like to use the Strichartz estimates for the nonlinear terms pathwisely. On the other hand, Proposition \ref{StochStrichartz} only gives us an $L^p(\Omega)$-estimate for the stochastic term at hand. Hence, it is natural to truncate the nonlinearities and look for a mild solution of
 \begin{equation}\label{ProblemStratonovichApprox}
 \left\{
 \begin{aligned}
 \df u_n(t)&= \left(\im \Delta u_n(t)-\im \lambda\varphi_n(u_n,t) \vert u_n(t) \vert^{\alpha-1} u_n(t)-\varphi_n(u_n,t)\mu_1(\vert u_n(t) \vert^{2(\gamma-1)} u_n(t))-\mu_2\left(u_n(t)\right)\right) \df t\\&\qquad-\im\left( \varphi_n(u_n,t) B_1 \left(  \vert u_n(t) \vert^{\gamma-1} u_n(t) \right) +B_2 u_n(t)\right)\df W(t),\\
 u(0)&=u_0.
 \end{aligned}\right.
 \end{equation}	
 in $\MFshort{T}$ for fixed $n\in\N.$ The truncation is given by
 \begin{align}\label{truncation}
  	\varphi_n(u,t):=\theta_n(\norm{u}_{L^{q}(0,t;L^{\alpha+1})}+\norm{u}_{L^{\tilde{q}}(0,t;L^{2 \gamma})}),\qquad
 \end{align}
 with 
 \begin{align*}
 	\theta_n(x):=\left\{
 	\begin{aligned}
 	&1, \hspace{1.8cm}x\in[0,n], \\
 	&2-\frac{x}{n},\hspace{0.95cm} x\in[n,2n],\\
 	&0, \hspace{1.8cm}x\in[2n,\infty).
 	\end{aligned}\right.
 \end{align*}
 In particular, we have 
 \begin{align}\label{cutoffLipschitz}
 \vert \theta_n(x)-\theta_n(y)\vert\le \frac{1}{n}\vert x-y\vert,\qquad x,y \ge 0.
 \end{align}
Before we start with the fixed point argument, we formulate properties of the power-type nonlinearities that appear in \eqref{ProblemStratonovich}.
\begin{Lemma}\label{FrechetNonlinear}
	Let $\left(S,\mathcal{A},\mu\right)$ be a measure space and $1\le \sigma<r<\infty$  Then, the map
		\begin{align*}
		G: L^r(S) \rightarrow L^{\frac{r}{\sigma}}(S), \hspace{0.7cm}G(u):=\vert u\vert^{\sigma-1}u,
		\end{align*}
		is continuously Fr\'{e}chet differentiable with 
		\begin{align*}
		\Vert G'[u]\Vert_{L^r \to L^{\frac{r}{\sigma}}}\lesssim \Vert u \Vert_{L^r}^{\sigma-1},\qquad u\in L^r(S). 
		\end{align*}
		In particular,
		\begin{align}\label{nonlinerityLocalLipschitz}
			\norm{G(u)-G(v)}_{L^\frac{r}{\sigma}}\lesssim \left(\norm{u}_{L^r}+\norm{v}_{L^r}\right)^{\sigma-1}\norm{u-v}_{L^r},\qquad u,v\in L^r(S).
		\end{align}
\end{Lemma}

\begin{proof}
	This Lemma is well known, see for example the lecture notes \cite{ISEM}, Lemma 9.2.
\end{proof}

To simplify the presentation, we use the following abbreviations for $r>0$ and $t\in [0,T]:$
\begin{align}\label{ConvolutionNonlinear}
	K_{det}^n u(t):=&-\im \lambda\int_0^t U(t-s)\left[\varphi_n(u,s) \vert u(s)\vert^{\alpha-1} u(s)\right]\df s,
\end{align}
\begin{align}\label{ConvolutionStrat}
K_{Strat}^n u(t):=&\int_0^t U(t-s)\left[\mu_1\left(\varphi_n(u,s)\vert u(s) \vert^{2(\gamma-1)} u(s)\right)+\mu_2\left( u(s)\right)\right]\df s,
\end{align}
\begin{align}\label{ConvolutionStoch}
	 K_{stoch}^n u(t):=&-\im\int_0^t U(t-s)\left[  B_1 \left( \varphi_n(u,s) \vert u(s) \vert^{\gamma-1} u(s) \right)+B_2 u(s)\right] \df W(s). 
\end{align}
Before we start with the proof of the local existence and uniqueness result in the subcritical case, we introduce our notion of a solution of \eqref{ProblemStratonovichApprox}.

\begin{Definition}
	Let $\alpha\in (1,1+\frac{4}{d}],$ $\gamma \in (1,1+\frac{2}{d}]$ and $p\in (1,\infty).$
	\begin{enumerate}
		\item[a)]
		A \emph{local mild solution} of \eqref{ProblemStratonovichApprox} is a  pair $\left(u^n,\tau^n\right)$ consisting of 
a stopping time $\tau^n\in [0,T]$ and a process $u^n\in \MFshort{\tau^n},$
		such that
		the equation
	\begin{align}\label{equationXn}
	u^n=U(\cdot)u_0+K_{det}^n u^n+ K_{Strat}^n u^n+K_{stoch}^n u^n
	\end{align}
		holds almost surely on $\left\{t\le \tau^n\right\}.$ 
		\item[b)] Solutions of \eqref{ProblemStratonovichApprox} are called \emph{unique}, if we have 
		\begin{align*}
		\Prob\left(u_1^n(t)=u_2^n(t)\quad \forall t\in [0,\sigma^n\land \tau^n)\right)=1
		\end{align*}
		for all local mild solutions $(u_1^n,\sigma^n)$ and $(u_2^n,\tau^n).$
		\item[c)] 
		A local mild solution $\left(u^n,\tau^n\right)$ with $\tau^n=T$ almost surely is called \emph{global mild solution}.
	\end{enumerate}
\end{Definition}

In the following Proposition, we state existence and uniqueness for \eqref{ProblemStratonovichApprox} in the subcritical case.


\begin{Prop}\label{cutoffEquationLzwei}
	Let $\alpha\in (1,1+\frac{4}{d}),$ $\gamma \in (1,1+\frac{2}{d})$ and $p\in(1,\infty).$  Then, there is a unique global mild solution $\left(u^n,T\right)$ of \eqref{ProblemStratonovichApprox}.
\end{Prop}

\begin{proof}
We fix $n\in\N$  and construct the solution from the assertion inductively. \\
\emph{Step 1:}
 We look for a fixed point of the operator 
given by
\begin{align*}
	K^n u:=U(\cdot)u_0+K_{det}^n u+K_{Strat}^n u+K_{stoch}^n u,\qquad u\in \FpSpaceIni,
\end{align*}
 where $r>0$ will be chosen small enough. 
 Let $u\in \FpSpaceIni.$ 
 A pathwise application of Proposition \ref{DetStrichartz} and integration over $\Omega$ yields 
\begin{align*}
	\norm{U(\cdot)u_0}_\FpSpaceIniNorm\lesssim \norm{u_0}_{L^2}.
\end{align*}
We define a stopping time by 
\begin{align*}
	\tau:=\inf \left\{t\ge 0: \norm{u}_{L^q(0,t;L^{\alpha+1})}+ \norm{u}_{L^{\tilde{q}}(0,t;L^{2\gamma})} \ge 2n\right\}\land r
\end{align*}
and set 
\begin{align*}
	\delta:=1+\frac{d}{4}(1-\alpha)>0,\qquad \tilde{\delta}=1+\frac{d}{2}(1-\gamma).
\end{align*}
We estimate 
\begin{align*}
	\norm{K_{det}^n u}_{E_r}\lesssim& \norm{\cutoffUdot \vert u\vert^{\alpha-1}u}_{L^{q'}(0,r;L^{\frac{\alpha+1}{\alpha}})}
	\le \norm{ \vert u\vert^{\alpha-1}u}_{L^{q'}(0,\tau;L^{\frac{\alpha+1}{\alpha}})}\le\norm{u}_{L^q(0,\tau;L^{\alpha+1})}^\alpha \tau^\delta
	\le \left(2 n\right)^\alpha r^\delta
\end{align*}
using Proposition \ref{DetStrichartz} $b)$ and $d)$ and the H\"older inequality. In the same spirit, we get 
\begin{align*}
	\norm{K_{Strat}^n u}_{E_r}\lesssim& \bigNorm{ \mu_1\left(\cutoffUdot \vert u\vert^{2(\gamma-1)}u\right)}_{L^{\tilde{q}'}(0,r;L^{\frac{2\gamma}{2\gamma-1}})}+\norm{\mu_2(u)}_{L^1(0,r;L^2)}\\
	\le& \frac{1}{2}\sumM \norm{e_m}_\LInfty^2\norm{ \vert u\vert^{2(\gamma-1)}u}_{L^{\tilde{q}'}(0,\tau;L^{\frac{2\gamma}{2\gamma-1}})}+\frac{1}{2}\sumM \norm{B_m}_{\mathcal{L}(L^2)}^2 r \norm{u}_{L^\infty(0,r;L^2)}\\
	\lesssim& \norm{u}_{L^{\tilde{q}}(0,\tau;L^{2\gamma})}^{2\gamma-1} \tau^{\tilde{\delta}} +r \norm{u}_{L^\infty(0,r;L^2)}
	\le \left(2 n\right)^{2\gamma-1} r^{\tilde{\delta}}+r \norm{u}_{L^\infty(0,r;L^2)}.
\end{align*}
Integrating over $\Omega$ yields 
\begin{align*}
	\norm{K_{det}^n u}_\FpSpaceIniNorm \lesssim \left(2 n\right)^\alpha r^\delta,\qquad 
	\norm{K_{Strat}^n u}_\FpSpaceIniNorm\lesssim \left(2 n\right)^{2\gamma-1} r^{\tilde{\delta}}+r\norm{ u}_\FpSpaceIniNorm.
\end{align*}
By Proposition \ref{StochStrichartz}, we obtain
\begin{align*}
&\norm{K_{stoch}^n u}_\FpSpaceIniNorm\lesssim \norm{B_1 \left(\cutoffUdot\vert u\vert^{\gamma-1}u\right)+B_2(u)}_{L^p(\Omega, L^2(0,r;\HS(Y,L^2)))}\\
&\hspace{0.5cm}\le \left(\sumM \norm{e_m}_{{\LInfty}}^2\right)^{\frac{1}{2}} \norm{\cutoffUdot\vert u\vert^{\gamma-1}u}_{L^p(\Omega,L^2(0,r;L^2))}
+\left(\sumM \norm{B_m}_{\mathcal{L}(L^2)}^2\right)^{\frac{1}{2}}\norm{u}_{L^p(\Omega,L^2(0,r;L^2))}\\
&\hspace{0.5cm}\lesssim  \norm{\cutoffUdot\vert u\vert^{\gamma-1}u}_{L^p(\Omega,L^2(0,r;L^2))}
+r^\frac{1}{2}\norm{u}_{L^p(\Omega,L^\infty(0,r;L^2))}
\end{align*}
From the pathwise inequality
\begin{align*}
	\norm{\cutoffUdot\vert u\vert^{\gamma-1}u}_{L^2(0,r;L^2)}
	\le \norm{u}_{L^{2\gamma}(0,\tau;L^{2\gamma})}^\gamma
	\le \tau^{\frac{\tilde{\delta}}{2}}\norm{u}_{L^{\tilde{q}}(0,\tau;L^{2\gamma})}^\gamma
	\le r^{\frac{\tilde{\delta}}{2}} (2n)^{\gamma}
\end{align*}
we conclude
\begin{align*}
	\norm{K_{stoch}^n u}_\FpSpaceIniNorm\lesssim r^{\frac{\tilde{\delta}}{2}} (2n)^{\gamma}+r^\frac{1}{2} \norm{u}_\FpSpaceIniNorm
\end{align*}
and altogether,
\begin{align*}
	\norm{K^n u}_\FpSpaceIniNorm \lesssim \norm{u_0}_\Lzwei+(2n)^\alpha r^\delta+(2n)^{2\gamma-1} r^{\tilde{\delta}}+ r^{\frac{\tilde{\delta}}{2}} (2n)^{\gamma}+\left(r+r^\frac{1}{2}\right) \norm{u}_\FpSpaceIniNorm<\infty
\end{align*}
for $u\in \FpSpaceIni$ and therefore the invariance of $\FpSpaceIni$ under $K^n.$ 
To show the contractivity of $K^n,$ we take $u_1,u_2\in \FpSpaceIni$ and define stopping times 
\begin{align*}
\tau_j:=\inf \left\{t\ge 0: \norm{u_j}_{L^q(0,t;L^{\alpha+1})}+ \norm{u_j}_{L^{\tilde{q}}(0,t;L^{2\gamma})}\ge 2n\right\}\land r,\qquad j=1,2,
\end{align*}
 and fix $\omega\in\Omega.$ Without loss of generality, we assume $\tau_1(\omega)\le \tau_2(\omega).$ We use the deterministic Strichartz estimates from Proposition \ref{DetStrichartz} 
\begin{align*}
	\norm{K_{det}^n(u_1)-K_{det}^n(u_2)}_{E_r}
	\lesssim& \norm{\cutoffUEinsDot \vert u_1\vert^{\alpha-1}u_1-\cutoffUZweiDot \vert u_2\vert^{\alpha-1}u_2}_{L^{q'}(0,r;L^{\frac{\alpha+1}{\alpha}})}\\
	\le&  \norm{\cutoffUEinsDot \left(\vert u_1\vert^{\alpha-1}u_1- \vert u_2\vert^{\alpha-1}u_2\right)}_{L^{q'}(0,r;L^{\frac{\alpha+1}{\alpha}})}\\
	&+ \norm{\left[\cutoffUEinsDot-\cutoffUZweiDot\right] \vert u_2\vert^{\alpha-1}u_2}_{L^{q'}(0,r;L^{\frac{\alpha+1}{\alpha}})}.
\end{align*}
By \eqref{cutoffLipschitz} and Lemma \ref{interpolation}, we get
\begin{align}\label{cutoffLipschitzProof}
	\left\vert \varphi_n(u_1,s)-\varphi_n(u_2,s)\right\vert
	&\le \frac{1}{n} \left\vert \norm{u_1}_{L^{q}(0,s;L^{\alpha+1})}+\norm{u_1}_{L^{\tilde{q}}(0,s;L^{2 \gamma})}-\norm{u_2}_{L^{q}(0,s;L^{\alpha+1})}-\norm{u_2}_{L^{\tilde{q}}(0,s;L^{2 \gamma})}\right\vert\nonumber\\
	&\le \frac{1}{n} \left(\norm{u_1-u_2}_{L^{q}(0,s;L^{\alpha+1})}+\norm{u_1-u_2}_{L^{\tilde{q}}(0,s;L^{2 \gamma})}\right)\nonumber\\
	&\le \frac{2}{n} \norm{u_1-u_2}_{E_s}
\end{align}
and we can use this as well as Lemma \ref{FrechetNonlinear} with $p=\alpha+1$ and $\sigma=\alpha$ to derive
\begin{align*}
	&\norm{\cutoffUEinsDot \left(\vert u_1\vert^{\alpha-1}u_1- \vert u_2\vert^{\alpha-1}u_2\right)}_{L^{q'}(0,r;L^{\frac{\alpha+1}{\alpha}})}\le \norm{\vert u_1\vert^{\alpha-1}u_1- \vert u_2\vert^{\alpha-1}u_2}_{L^{q'}(0,\tau_1;L^{\frac{\alpha+1}{\alpha}})}\\
	&\hspace{2cm}\le \tau_1^\delta \left(\norm{u_1}_{L^q(0,\tau_1,L^{\alpha+1})}+\norm{u_2}_{L^q(0,\tau_1,L^{\alpha+1})}\right)^{\alpha-1} \norm{u_1-u_2}_{L^q(0,\tau_1,L^{\alpha+1})}\\
	&\hspace{2cm}\le r^\delta (4n)^{\alpha-1} \norm{u_1-u_2}_{L^q(0,\tau_1,L^{\alpha+1})}
	\le r^\delta (4n)^{\alpha-1} \norm{u_1-u_2}_{E_r}
\end{align*}
and
\begin{align*}
	&\norm{\left[\cutoffUEinsDot-\cutoffUZweiDot\right] \vert u_2\vert^{\alpha-1}u_2}_{L^{q'}(0,r;L^{\frac{\alpha+1}{\alpha}})}
	\le \frac{2}{n} \bigNorm{ \norm{u_1-u_2}_{E_\cdot} \vert u_2\vert^{\alpha-1}u_2}_{L^{q'}(0,\tau_2;L^{\frac{\alpha+1}{\alpha}})}\\
	&\hspace{2cm}\le \frac{2}{n} \norm{u_1-u_2}_{E_r} \norm{\vert u_2\vert^{\alpha-1}u_2}_{L^{q'}(0,\tau_2;L^{\frac{\alpha+1}{\alpha}})}\\
	&\hspace{2cm}\le \frac{2}{n} \norm{u_1-u_2}_{E_r} \tau_2^\delta \norm{u_2}_{L^{q}(0,\tau_2;L^{\alpha+1})}^\alpha
	\le \frac{2}{n} \norm{u_1-u_2}_{E_r} r^\delta (2n)^{\alpha}.
\end{align*}
We obtain
\begin{align*}
\norm{K_{det}^n(u_1)-K_{det}^n(u_2)}_{E_r}
\lesssim& \left(2^{\alpha+1}+4^{\alpha-1}\right)r^\delta n^{\alpha-1} \norm{u_1-u_2}_{E_r}.
\end{align*}
Analogously,
\begin{align*}
\norm{K_{Strat}^n(u_1)-K_{Strat}^n(u_2)}_{E_r}
\lesssim& \norm{\cutoffUEinsDot \vert u_1\vert^{\alpha-1}u_1-\cutoffUZweiDot \vert u_2\vert^{\alpha-1}u_2}_{L^{\tilde{q}'}(0,r;L^{\frac{2\gamma}{2\gamma-1}})}\\&+\norm{u_1-u_2}_{L^1(0,r;L^2)}\\
\lesssim& \left[\left(2^{2\gamma}+4^{2(\gamma-1)}\right)r^{\tilde{\delta}} n^{2(\gamma-1)}+r\right] \norm{u_1-u_2}_{L^{\tilde{q}}(0,r;L^{2\gamma})}\\
\le& \left[\left(2^{2\gamma}+4^{2(\gamma-1)}\right)r^{\tilde{\delta}} n^{2(\gamma-1)}+r\right] \norm{u_1-u_2}_{E_r}.
\end{align*}
 For the stochastic convolution, we estimate
\begin{align*}
	\norm{\cutoffUEinsDot\vert u_1\vert^{\gamma-1}u_1-\cutoffUZweiDot\vert u_2\vert^{\gamma-1}u_2}_{L^2(0,r;L^2)}
	\lesssim& 	\norm{\cutoffUEinsDot \left(\vert u_1\vert^{\gamma-1}u_1-\vert u_2\vert^{\gamma-1}u_2\right)}_{L^2(0,r;L^2)}\\
	&+\norm{(\cutoffUEinsDot-\cutoffUZweiDot) \vert u_2\vert^{\gamma-1}u_2}_{L^2(0,r;L^2)}.
\end{align*}
The terms on the RHS can be treated by Lemma \ref{FrechetNonlinear} with $p=2\gamma$ and $\sigma=\gamma$ and Lemma \ref{interpolation}
\begin{align*}
	&\norm{\cutoffUEinsDot \left(\vert u_1\vert^{\gamma-1}u_1-\vert u_2\vert^{\gamma-1}u_2\right)}_{L^2(0,r;L^2)}\\
	&\hspace{4cm}\lesssim \left(\norm{u_1}_{L^{2\gamma}(0,\tau_1;L^{2\gamma})}+\norm{u_2}_{L^{2\gamma}(0,\tau_1;L^{2\gamma})}\right)^{\gamma-1}	\norm{ u_1-u_2}_{L^{2\gamma}(0,\tau_1;L^{2\gamma})}\\
	&\hspace{4cm}\lesssim \tau_1^{\frac{\tilde{\delta}}{2}} \left(\norm{u_1}_{L^{\tilde{q}}(0,\tau_1;L^{2\gamma})}+\norm{u_2}_{L^{\tilde{q}}(0,\tau_1;L^{2\gamma})}\right)^{\gamma-1}	\norm{ u_1-u_2}_{L^{\tilde{q}}(0,\tau_1;L^{2\gamma})}\\
		&\hspace{4cm}\lesssim r^{\frac{\tilde{\delta}}{2}} \left(4 n\right)^{\gamma-1}	\norm{ u_1-u_2}_{E_r}
\end{align*}
and by the estimate \eqref{cutoffLipschitzProof}
\begin{align*}
	\norm{(\cutoffUEinsDot-\cutoffUZweiDot) \vert u_2\vert^{\gamma-1}u_2}_{L^2(0,r;L^2)}
	&\le \frac{2}{n} \norm{u_1-u_2}_{E_r} \norm{\vert u_2\vert^{\gamma-1}u_2}_{L^2(0,\tau_2;L^2)}\\
			&\le \frac{2}{n} \norm{u_1-u_2}_{E_r} \tau_2^{\frac{\tilde{\delta}}{2}} \norm{u_2}_{L^{\tilde{q}}(0,\tau_2;L^{2\gamma})}^\gamma\\
			&\le \frac{2}{n} \norm{u_1-u_2}_{E_r} r^{\frac{\tilde{\delta}}{2}} \left(2n\right)^\gamma.
\end{align*}
By Proposition \ref{StochStrichartz}, this yields
\begin{align*}
\norm{K_{stoch}^n(u_1)-K_{stoch}^n(u_2)}_\FpSpaceIniNorm
&\lesssim
	\norm{B_1(\cutoffUEinsDot\vert u_1\vert^{\gamma-1}u_1-\cutoffUZweiDot\vert u_2\vert^{\gamma-1}u_2)}_{L^q(\Omega,L^2(0,r;\HS(Y,L^2)))}\\
	&\hspace{1cm}+\norm{B_2(u_1-u_2)}_{L^q(\Omega,L^2(0,r;\HS(Y,L^2)))}\\
	&\le \left(\sumM \norm{e_m}_{L^\infty}^2\right)^\frac{1}{2}r^{\frac{\tilde{\delta}}{2}} n^{\gamma-1} \left(4^{\gamma-1} +2^{\gamma+1}\right)	\norm{u_1-u_2}_\FpSpaceIniNorm\\
	&\hspace{2cm}+\left(\sumM \norm{B_m}_{\mathcal{L}(L^2)}^2\right)^\frac{1}{2} r^\frac{1}{2}\norm{u_1-u_2}_\FpSpaceIniNorm\\
	&\lesssim \left[r^{\frac{\tilde{\delta}}{2}} n^{\gamma-1} \left(4^{\gamma-1} +2^{\gamma+1}\right)+r^\frac{1}{2}\right]	\norm{u_1-u_2}_\FpSpaceIniNorm.
\end{align*}
Collecting the estimates for the other terms leads to
\begin{align}\label{contractionEstimate}
\norm{K^n(u_1)-K^n(u_2)}_\FpSpaceIniNorm
\lesssim& \Big[\left(2^{\alpha+1}+4^{\alpha-1}\right)r^\delta n^{\alpha-1}+\left(2^{2\gamma}+4^{2(\gamma-1)}\right)r^{\tilde{\delta}} n^{2(\gamma-1)}+r\nonumber\\
&\hspace{3cm}+ \left(4^{\gamma-1} +2^{\gamma+1}\right)r^{\frac{\tilde{\delta}}{2}} n^{\gamma-1}+r^\frac{1}{2}\Big] \norm{u_1-u_2}_\FpSpaceIniNorm.
\end{align}
Hence, there is a small time $r=r(n,\alpha,\gamma)>0,$ such that $K^n$ is a strict contraction in $\FpSpaceIni$ with Lipschitz constant $\le \frac{1}{2}$ and Banach's Fixed Point Theorem yields $u^{n,1}\in \FpSpaceIni$ with $K^n(u_1^n)=u_1^n.$ \\

\emph{Step 2:}
We choose $r>0$ as in the first step and
 assume that we have $k\in \N$ and $u_k^n\in \MFshort{kr}$ with
\begin{align*}
u_k^n=U(\cdot)u_0+K_{det}^n u_k^n+ K_{Strat}^n u_k^n+K_{stoch}u_k^n
\end{align*}
on the interval $[0,kr].$ In order to extend $u_k^n$ to $[kr,(k+1)r]$, we define a new cutoff function by $\varphi_{n,k}(u,t):=\theta_n\left(Z_t(u)\right),$ where $\left(Z_t(u)\right)_{t\in[0,r]}$ is a continuous, $\Filtration^{kr}$-adapted process given by
\begin{align*}
	  	Z_t(u):=&(\norm{u_k^n}_{L^{q}(0,kr;L^{\alpha+1})}^q+\norm{u}_{L^{q}(0,t;L^{\alpha+1})}^q)^\frac{1}{q}+(\norm{u_k^n}_{L^{\tilde{q}}(0,kr;L^{2\gamma})}^{\tilde{q}}+\norm{u}_{L^{\tilde{q}}(0,t;L^{2\gamma})}^{\tilde{q}})^\frac{1}{\tilde{q}}
\end{align*}
for $t\in [0,r]$ and $u\in\MFshortShifted{kr}{r}.$ Moreover, we set
\begin{align*}
K_{det,k}^n u(t):=-\im \lambda\int_0^t U(t-s)\left[ \varphi_{n,k}(u,s) \vert u(s)\vert^{\alpha-1} u(s)\right]\df s,
\end{align*}
\begin{align*}
K_{Strat,k}^n u(t):=&\int_0^t U(t-s)\left[\varphi_{n,k}(u,s)\mu_1\left(\vert u(s) \vert^{2(\gamma-1)} u(s)\right)+\mu_2\left(u(s)\right)\right]\df s,
\end{align*}
\begin{align*}
	K_{stoch,k}^n u(t):= -\im \int_0^t U(t-s) \left[\varphi_{n,k}(u,s)B_1\left(\vert u(s)\vert^{\gamma-1} u(s)\right)+B_2 u(s)\right]\df W^{kr}(s)
\end{align*}
for $t\in [0,r]$ and $u\in\MFshortShifted{kr}{r}$ and 
\begin{align*}
	K_k^n u:=U(\cdot)u_k^n(kr)+K_{det,k}^n u+K_{Strat,k}^n u+K_{stoch,k}^n u,\qquad  u\in\MFshortShifted{kr}{r},
\end{align*}
We take $v_1,v_2 \in \MFshortShifted{kr}{r}$ and define the $\F^{kr}$-stopping times
\begin{align}\label{stoppingTimeInductionStepLzwei}
\tau_j:=\inf \left\{t\ge 0: Z_t(v_j)\ge 2n\right\}\land r,\qquad j=1,2.
\end{align}
Without loss of generality, we assume $\tau_1(\omega)\le \tau_2(\omega)$ and follow the lines of the initial step where we replace $u_j$ by $v_j$ and $\varphi_n(u_j)$ by  $\varphi_{n,k}(v_j)$ for $j=1,2.$
We obtain
\begin{align*}
	\norm{K_{det,k}^n v_1-K_{det,k}^n v_2}_{E_r}\le& \tau_1^\delta \left(\norm{v_1}_{L^q(0,\tau_1,L^{\alpha+1})}+\norm{v_2}_{L^q(0,\tau_1,L^{\alpha+1})}\right)^{\alpha-1} \norm{v_1-v_2}_{E_r}\\
	&+ \frac{2}{n} \norm{v_1-v_2}_{E_r} \tau_2^\delta \norm{v_2}_{L^{q}(0,\tau_2;L^{\alpha+1})}^\alpha
\end{align*}
and by
	$\norm{v_j}_{L^q(0,\tau_1,L^{\alpha+1})}\le Z_{\tau_1}(v_j)\le 2 n$
for $j=1,2,$ we conclude
\begin{align*}
\norm{K_{det,k}^n v_1-K_{det,k}^n v_2}_{E_r}
&\le \tau_1^\delta \left(4 n\right)^{\alpha-1} \norm{v_1-v_2}_{E_r}+ \frac{2}{n} \norm{v_1-v_2}_{E_r} \tau_2^\delta \left(2 n\right)^\alpha \\
&\le r^\delta \left( \left(4 n\right)^{\alpha-1}+ \frac{2}{n}  \left(2 n\right)^\alpha\right)\norm{v_1-v_2}_{E_r}.
\end{align*}
Analogously, the estimates for $K_{Strat}^n$ and $K_{stoch}^n$ from the first step can be adapted to get  
\begin{align*}
\norm{K_{Strat,k}^n(v_1)-K_{Strat,k}^n(v_2)}_{E_r}
\lesssim& \left(\left(2^{2\gamma}+4^{2(\gamma-1)}\right)r^{\tilde{\delta}} n^{2(\gamma-1)} +r\right)\norm{v_1-v_2}_{E_r},
\end{align*}
\begin{align*}
\norm{K_{stoch,k}^n(v_1)-K_{stoch,k}^n(v_2)}_{\mathbb{M}_{\Filtration^{kr}}^p(\Omega,E_r)}
&\lesssim \left(r^{\frac{\tilde{\delta}}{2}} n^{\gamma-1} \left(4^{\gamma-1} +2^{\gamma+1}\right)+r^\frac{1}{2}\right)	\norm{v_1-v_2}_{\mathbb{M}_{\Filtration^{kr}}^p(\Omega,E_r)}
\end{align*}
and thus
\begin{align}\label{contractionEstimateInduction}
\norm{K_k^n(v_1)-K_k^n(v_2)}_{\mathbb{M}_{\Filtration^{kr}}^p(\Omega,E_r)}
\lesssim& \Big[\left(2^{\alpha+1}+4^{\alpha-1}\right)r^\delta n^{\alpha-1}+\left(2^{2\gamma}+4^{2(\gamma-1)}\right)r^{\tilde{\delta}} n^{2(\gamma-1)}+r\nonumber\\
&\hspace{2.5cm}+ \left(4^{\gamma-1} +2^{\gamma+1}\right)r^{\frac{\tilde{\delta}}{2}} n^{\gamma-1}+r^\frac{1}{2}\Big] \norm{v_1-v_2}_{\mathbb{M}_{\Filtration^{kr}}^p(\Omega,E_r)}.
\end{align}
Since the constant is the same as in the initial step, the definition of $r>0$ yields that $K_k^n$ is a strict contraction in $\MFshortShifted{kr}{r}.$ We call the unique fixed point  $v_{k+1}^n$ and set
	\begin{equation*}
	u_{k+1}^n(t):= \left\{ \begin{aligned} 
	&u_k^n(t), \hspace{2.25cm}t\in [0,kr],\\ &v_{k+1}^n(t-kr), \hspace{1cm} t\in[k r,(k+1)r].
	\end{aligned}\right.
	\end{equation*}
Obviously, $u_{k+1}^n$ is a continuous $\Filtration$-adapted process with $\norm{u_{k+1}^n}_{L^p(\Omega,E_{(k+1)r})}<\infty$ and therefore $u_{k+1}^n\in \MFshort{(k+1)r}.$
Let $t\in [k r,(k+1)r]$ and define $\tilde{t}:=t-k r.$ Then, the definition of $K_k^n$ and the induction assumption yield
\begin{align*}
	u_{k+1}^n(t)=&v_{k+1}^n(\tilde{t})=K_k^n v_{k+1}^n(\tilde{t})
	=U(\tilde{t})u_k^n(kr)+K_{det,k}^n v_{k+1}^n(\tilde{t})+K_{Strat,k}^n v_{k+1}^n(\tilde{t})+K_{stoch,k}^n v_{k+1}^n(\tilde{t})\\
	=&U(t)u_0+
	\left[U(\tilde{t})K_{det}^n u_k^n (kr)+K_{det,k}^n v_{k+1}^n(\tilde{t})\right]+\left[U(\tilde{t}) K_{Strat,k}^n u_k^n(kr)+K_{Strat,k}^n v_{k+1}^n(\tilde{t})\right]\\
	&+\left[U(\tilde{t})K_{stoch}u_k^n(r)+K_{stoch,k}^n v_{k+1}^n(\tilde{t})\right].
\end{align*}
Using the identities
\begin{align*}
	\varphi_n(u_k^n,s)=\varphi_n(u_{k+1}^n,s),\qquad 
	\varphi_{n,k}(v_{k+1}^n,\tilde{s})=\varphi_n(u_{k+1}^n,kr+\tilde{s})
\end{align*}
for $s\in [0,kr]$ and $\tilde{s}\in [0,r]$, 
we compute
\begin{align*}
	U(\tilde{t})K_{det}^n u_k^n (kr)+K_{det,k}^n& v_{k+1}^n(\tilde{t})=
	-\im \lambda U(\tilde{t})\int_0^{kr} U(k r -s)\left[ \varphi_n(u_k^n,s)\vert u_k^n(s)\vert^{\alpha-1} u_k^n(s)\right]\df s\\
	&-\im \lambda \int_0^{\tilde{t}} U(\tilde{t}-\tilde{s})\left[ \varphi_{n,k}(v_{k+1}^n,\tilde{s}) \vert v_{k+1}^n(\tilde{s})\vert^{\alpha-1} v_{k+1}^n(\tilde{s})\right]\df \tilde{s}\\
	=&
	-\im \lambda \int_0^{kr} U(t-s)\left[ \varphi_n(u_{k+1}^n,s) \vert u_{k+1}^n(s)\vert^{\alpha-1} u_{k+1}^n(s)\right]\df s\\
	&-\im \lambda \int_0^{\tilde{t}} U(\tilde{t}-\tilde{s})\left[ \varphi_n(u_{k+1}^n,kr+\tilde{s}) \vert u_{k+1}^n(kr+\tilde{s})\vert^{\alpha-1} u_{k+1}^n(k r+\tilde{s})\right]\df \tilde{s}\\
	=&
	-\im \lambda \int_0^{t} U(t-s)\left[ \varphi_n(u_{k+1}^n,s) \vert u_{k+1}^n(s)\vert^{\alpha-1} u_{k+1}^n(s)\right]\df s=K_{det}^n u_{k+1}^n(t),
\end{align*}
where we used the substitution $s=k r + \tilde{s}$ in the second integral for the last step.
Analogously,
\begin{align*}
	U(\tilde{t})K_{Strat,k}^n u_k^n (kr)+K_{Strat,k}^n v_{k+1}^n(\tilde{t})
	=K_{Strat}^n u_{k+1}^n(t),
\end{align*}
\begin{align*}
U(\tilde{t})K_{stoch}^n u_k^n (kr)+K_{Stoch,k} v_{k+1}^n(\tilde{t})
=K_{stoch}^n u_{k+1}^n(t),
\end{align*}
where one uses \eqref{shiftStochIntegral} for the stochastic convolutions.
Hence, we get 
\begin{align*}
u_{k+1}^n(t)
=U(t)u_0+
K_{det}^n u_{k+1}^n(t)+K_{Strat}^n u_{k+1}^n(t)
+K_{stoch}^n u_{k+1}^n(t)=K^n u_{k+1}^n(t) 
\end{align*}
for $t\in [kr,(k+1)r]$ and therefore, 
$u_{k+1}^n$ is a fixed point of $K^n$ in $\MFshort{(k+1)r}.$ 
Define $k:= \lfloor \frac{T}{r}+1 \rfloor.$ Then, $u^n:=u_{k}^n$ is the process from the assertion.

\emph{Step 3:}
Now, we turn our attention to uniqueness. Let $(\tilde{u},\tau)$ be another local mild solution of \eqref{ProblemStratonovichApprox}. As in \eqref{contractionEstimate}, we get 
\begin{align*}
\norm{u-\tilde{u}}_{\MFshort{\tau\land r}}=&\norm{K^n(u)-K^n(\tilde{u})}_{\MFshort{\tau\land r}}\\
\le& C \Big[\left(2^{\alpha+1}+4^{\alpha-1}\right)r^\delta n^{\alpha-1}+\left(2^{2\gamma}+4^{2(\gamma-1)}\right)r^{\tilde{\delta}} n^{2(\gamma-1)}+r\nonumber\\
&\hspace{3cm}+ \left(4^{\gamma-1} +2^{\gamma+1}\right)r^{\frac{\tilde{\delta}}{2}} n^{\gamma-1}+r^\frac{1}{2}\Big] \norm{u-\tilde{u}}_{\MFshort{\tau\land r}}\\
\le& \frac{1}{2}\norm{u-\tilde{u}}_{\MFshort{\tau\land r}},
\end{align*} 
which leads to $u(t)=\tilde{u}(t)$ in $\MFshort{\tau\land r},$ i.e. $u=\tilde{u}$ almost surely on $\left\{t\le \tau\land r\right\}.$
 This can be iterated to see that $u(t)=\tilde{u}(t)$  almost surely on $\left\{t\le \sigma_k\right\}$ with $\sigma_k:=\tau \land (k r)$ for $k\in\N.$ The assertion follows from $\sigma_k =\tau$ for $k$ large enough. 
\end{proof}

%
%
%
%

In the following two Propositions,  we use the results on the truncated equation \eqref{ProblemStratonovichApprox} to derive existence and uniqueness for the original problem \eqref{ProblemStratonovich}. The proofs are quite standard and in the literature, analogous arguments have been used in various contexts for extensions of existence and uniqueness results from integrable to non-integrable initial values and from globally to locally Lipschitz nonlinearities, see for example \cite{UMDstochEvo}, Theorem 7.1, \cite{BrzezniakConvolutions}, Theorem 4.10, and \cite{Seidler1993}, Theorem 1.5.

\begin{Prop}\label{LocalExistenceLzwei}
	Let  $\alpha \in (1,1+\frac{4}{d}),$ $\gamma \in (1,1+\frac{2}{d})$ and $\left(u^n\right)_{n\in\N}\subset \MFshort{T}$ be the sequence constructed in Proposition \ref{cutoffEquationLzwei}.
	For $n\in\N,$
	we define the stopping time $\tau_n$ by
	\begin{align*}
	\tau_n:=\inf \left\{t\in [0,T]: \norm{u^n}_{L^q(0,t;L^{\alpha+1})}+ \norm{u^n}_{L^{\tilde{q}}(0,t;L^{2\gamma})}\ge n\right\}\land T.
	\end{align*}
	Then, the following assertions hold: 
	\begin{enumerate}
		\item[a)] We have $0<\tau_n\le \tau_k$ almost surely for $n\le k$ and  $u^n(t)=u^k(t)$ almost surely on $\left\{t\le \tau_n \right\}.$ 
		\item[b)] The triple $\left(u,\left(\tau_n\right)_{n\in\N},\tau_\infty\right)$ with  $u(t):=u^n(t)$ for $t\in [0,\tau_n]$ and $\tau_\infty:=\sup_{n\in\N}\tau_n$ is a local mild solution of \eqref{ProblemStratonovich}.
	\end{enumerate}
\end{Prop}

\begin{proof}
\emph{ad a):} 
We note that $\tau_n$ is a welldefined stopping time with $\tau_n>0$ almost surely, since 
\begin{align*}
	Z^n(t):=\norm{u^n}_{L^q(0,t;L^{\alpha+1})}+ \norm{u^n}_{L^{\tilde{q}}(0,t;L^{2\gamma})}\le 2 \norm{u^n}_{E_t}\le 2\norm{u^n}_{E_T}<\infty,\qquad t\in[0,T],
\end{align*}
defines an increasing, continuous and $\Filtration$-adapted process $Z^n: \Omega\times [0,T]\to [0,\infty)$ with $Z^n(0)=0.$ For $n\le k,$ we set
\begin{align*}
	\tau_{k,n}:=\inf \left\{t\in [0,T]: Z^k(t)\ge n\right\}\land T.
\end{align*}
Then, we have $\tau_{k,n}\le \tau_k$ and $\varphi_n(u^k,t)=1=\varphi_k(u^k,t)$ on $\left\{t\le \tau_{k,n}\right\}.$ Hence, $(u^k,\tau_{k,n})$ is a solution of \eqref{ProblemStratonovichApprox} and by the uniqueness part of Proposition \ref{cutoffEquationLzwei}, we obtain $u^k(t)=u^n(t)$ almost surely on $\left\{t\le \tau_{k,n}\right\}.$ But this leads to $Z^k(t)=Z^n(t)$ on $\left\{t\le \tau_{k,n}\right\}$ and  $\tau_{k,n}=\tau_n$ almost surely which implies the assertion.\\

	\emph{ad b):}
	By part a), $u$ is welldefined up to a null set, where we define $u:=0$ and $\tau_\infty=T.$ The monotonicity of $\left(\tau_n\right)_{n\in\N}$ yields $\tau_n\to \tau_\infty$ almost surely. Moreover, $u\in \MFshort{\tau_n}$ by  Proposition \ref{cutoffEquationLzwei} and therefore $u\in \MFloc{\tau}.$	
 From \eqref{equationXn} and the identity
	\begin{align*}
	\varphi_n(u,t)=\varphi_n(u^n,t)=1\qquad \text{a.s on $\left\{t\land \tau_n\right\}$},
	\end{align*}
	we finally obtain
		\begin{align*}
		u(t)=&  U(t)u_0+ \int_0^{t} U(t-s)\left[-\im \lambda\vert u(s)\vert^{\alpha-1} u(s)+\mu_1\left(\vert u(s)\vert^{2(\gamma-1)}u(s)\right)+\mu_2(u(s))\right] \df s\\
		&\hspace{2cm}- \im \int_0^{t}  U(t-s)\left[B_1\left( \vert u(s)\vert^{\gamma-1} u(s)\right)+B_2u(s)\right] \df W(s)
		\end{align*}
	almost surely on $\left\{t\le \tau_n\right\}$ for all $n\in\N.$
\end{proof}

\begin{Prop}\label{PropUniqueness}
	Let  $\alpha \in (1,1+\frac{4}{d}),$ $\gamma \in (1,1+\frac{2}{d})$ and $\left(u_1,\left(\sigma_n\right)_{n\in\N},\sigma\right),$ $\left(u_2,\left(\tau_n\right)_{n\in\N},\tau\right)$ be local mild solutions to \eqref{ProblemStratonovich}.
	Then, 
	\begin{align*}
		u_1(t)=u_2(t) \qquad \text{a.s. on $\{t< \sigma\land \tau\}$},
	\end{align*}
	i.e. the solution of \eqref{ProblemStratonovich} is unique.
\end{Prop}

\begin{proof}
	We fix $k,n\in\N$ and define a stopping time by 
	\begin{align*}
	\nu_{k,n}:=&\inf \Big\{t\in [0,T]: \left(\norm{u_1}_{L^{q}(0,t;L^{\alpha+1})}+\norm{u_2}_{L^{\tilde{q}}(0,t;L^{2 \gamma})}\right)\\&\hspace{6cm}\lor\left( \norm{u_2}_{L^{q}(0,t;L^{\alpha+1})}+\norm{u_2}_{L^{\tilde{q}}(0,t;L^{2 \gamma})}\right)\ge n \Big\}\land \sigma_k\land \tau_k.
	\end{align*}
	Hence, $\varphi_n(u_1,t)=\varphi_n(u_2,t)=1$ on $\left\{t\le \nu_{k,n}\right\}$ and therefore, $\left(u_1, \nu_{k,n}\right)$ and
	$\left(u_2, \nu_{k,n}\right)$ are local mild solutions of \eqref{ProblemStratonovichApprox}. By the uniqueness part of Proposition \ref{cutoffEquationLzwei}, we get 
	\begin{align*}
	u_1(t)=u_2(t) \qquad \text{a.s. on $\{t\le \nu_{k,n}\}$},
	\end{align*}  	
	which yields the assertion, since $\nu_{k,n}\to \sigma\land \tau$ almost surely for $n,k\to \infty.$
\end{proof}

In the Propositions \ref{LocalExistenceLzwei} and \ref{PropUniqueness}, we have proved Theorem \ref{mainTheorem}, a) in the subcritical case, i.e. $\alpha \in (1,1+\frac{4}{d}),$ $\gamma \in (1,1+\frac{2}{d}).$ We continue with the critical setting.

\begin{Prop}
	Let $\alpha \in (1,1+\frac{4}{d}],$ $\gamma \in (1,1+\frac{2}{d}]$ with $\alpha=1+\frac{4}{d}$ or $\gamma=1+\frac{2}{d}.$ Then there is a unique local mild solution of \eqref{ProblemStratonovichBasics}.
\end{Prop}

\begin{proof}
	\begin{align*}
	\norm{K_1^\nu u}_\FpSpaceIniNorm \lesssim \norm{u_0}_\Lzwei+(2n)^\alpha r^\delta+(2n)^{2\gamma-1} r^{\tilde{\delta}}+ r^{\frac{\tilde{\delta}}{2}} (2n)^{\gamma}+\left(r+r^\frac{1}{2}\right) \norm{u}_\FpSpaceIniNorm<\infty
	\end{align*}
	\begin{align}\label{contractionEstimateCritical}
	\norm{K_1^\nu(u_1)-K^n(u_2)}_\FpSpaceIniNorm
	\lesssim& \Big[\left(2^{\alpha+1}+4^{\alpha-1}\right)r^\delta n^{\alpha-1}+\left(2^{2\gamma}+4^{2(\gamma-1)}\right)r^{\tilde{\delta}} n^{2(\gamma-1)}+r\nonumber\\
	&\hspace{3cm}+ \left(4^{\gamma-1} +2^{\gamma+1}\right)r^{\frac{\tilde{\delta}}{2}} n^{\gamma-1}+r^\frac{1}{2}\Big] \norm{u_1-u_2}_\FpSpaceIniNorm.
	\end{align}
	\emph{Step 1.}
	Let $\nu>0$ and $q:= 2+\frac{4}{d}.$ Then, $(q,q)$ is a Strichartz pair. For $r>0,$ we define
	\begin{align*}
	Y_r:=L^q(0,r;L^q(\Rd)),\qquad E_r:=C([0,r],\Lzwei)\cap Y_r
	\end{align*}
	and as in the proof of Proposition \ref{cutoffEquationLzwei}, we set
	\begin{align*}
	K_1^\nu u:=U(\cdot)u_0+K_{det}^\nu u+K_{Strat}^\nu u+K_{stoch}^\nu u
	\end{align*}
	with the convolution operators from \eqref{ConvolutionNonlinear}, \eqref{ConvolutionStrat} and \eqref{ConvolutionStoch}
	 and obtain the estimates 
	 	\begin{align*}
	 	\norm{K_1^\nu u}_\FpSpaceIniNorm \lesssim \norm{u_0}_\Lzwei+(2\nu)^\alpha r^\delta+(2\nu)^{2\gamma-1} r^{\tilde{\delta}}+ r^{\frac{\tilde{\delta}}{2}} (2\nu)^{\gamma}+\left(r+r^\frac{1}{2}\right) \norm{u}_\FpSpaceIniNorm
	 	\end{align*}
	 	\begin{align*}
	 	\norm{K_1^\nu(u_1)-K_1^\nu(u_2)}_\FpSpaceIniNorm
	 	\lesssim& \Big[\left(2^{\alpha+1}+4^{\alpha-1}\right)r^\delta \nu^{\alpha-1}+\left(2^{2\gamma}+4^{2(\gamma-1)}\right)r^{\tilde{\delta}} \nu^{2(\gamma-1)}+r\nonumber\\
	 	&\hspace{3cm}+ \left(4^{\gamma-1} +2^{\gamma+1}\right)r^{\frac{\tilde{\delta}}{2}} \nu^{\gamma-1}+r^\frac{1}{2}\Big] \norm{u_1-u_2}_\FpSpaceIniNorm.
	 	\end{align*}
	for $u,u_1,u_2\in \MFshort{r},$
	where we set 
	\begin{align*}
		\delta:=1+\frac{d}{4}(1-\alpha),\qquad \tilde{\delta}:=1+\frac{d}{2}(1-\gamma).
	\end{align*}
	 Note that replacing the integer $n$ by $\nu>0$ in the cutoff function does not change the estimates at all. Since we have $\delta=0$ or $\tilde{\delta}=0$ by the assumption, we cannot ensure that $K_1^\nu$ is a contraction by taking $r$ small enough. But if we choose $\nu$ and $r$ sufficiently small, we get a unique fixed point $u_1\in \MFshort{r}$  of $K_1^\nu.$\\
	By the definition of the truncation function $\varphi_\nu$ in \eqref{truncation}, $u_1$ is a solution of the original equation, as long as $\norm{u_1}_{L^q(0,t; L^q)}+\norm{u_1}_{L^{q_1}(0,t; L^{p_1})}\le \nu$ for 
	\begin{align*}
		p_1:=\left\{
		\begin{aligned}
		&2\gamma, \hspace{1.5cm} \alpha=1+\frac{4}{d}, \\
		&\alpha+1,\hspace{1cm} \gamma=1+\frac{2}{d},
		\end{aligned}\right.
	\end{align*}
	and $q_1>2$ such that $(p_1,q_1)$ is a Strichartz pair. 	
	Hence, the pair $(u_1,\tau_1)$ with 
	\begin{align*}
	\tau_1:=\inf \left\{t\ge 0: \norm{u_1}_{L^q(0,t;L^q)}+\norm{u_1}_{L^{q_1}(0,t; L^{p_1})}\ge \nu\right\}\land r
	\end{align*}
	is a local mild solution of \eqref{ProblemStratonovich}.\\

	 \emph{Step 2.} Next, we define the operator 
	 	\begin{align*}
	 	K_2^\nu u:=U(\cdot)u_1(\tau_1)+K_{det}^\nu u+K_{Strat}^\nu u+K_{stoch,2}^\nu u
	 	\end{align*}
	 	with 
\begin{align*}
K_{stoch,2}^\nu u(t):=&-\im\int_0^t U(t-s)\left[ \varphi_n(u,s) B_1 \left(  \vert u(s) \vert^{\gamma-1} u(s) \right)+B_2 u(s)\right] \df W^{\tau_1}(s)
\end{align*}	 	
and as above, we derive the estimates 
	 	\begin{align*}
	 	\norm{K_2^\nu u}_\FpSpaceIniNorm \lesssim \norm{u_1(\tau_1)}_\Lzwei+(2\nu)^\alpha r^\delta+(2\nu)^{2\gamma-1} r^{\tilde{\delta}}+ r^{\frac{\tilde{\delta}}{2}} (2\nu)^{\gamma}+\left(r+r^\frac{1}{2}\right) \norm{u}_\FpSpaceIniNorm
	 	\end{align*}
	 	\begin{align*}
	 	\norm{K_2^\nu(u_1)-K_2^\nu(u_2)}_\FpSpaceIniNorm
	 	\lesssim& \Big[\left(2^{\alpha+1}+4^{\alpha-1}\right)r^\delta \nu^{\alpha-1}+\left(2^{2\gamma}+4^{2(\gamma-1)}\right)r^{\tilde{\delta}} \nu^{2(\gamma-1)}+r\nonumber\\
	 	&\hspace{3cm}+ \left(4^{\gamma-1} +2^{\gamma+1}\right)r^{\frac{\tilde{\delta}}{2}} \nu^{\gamma-1}+r^\frac{1}{2}\Big] \norm{u_1-u_2}_\FpSpaceIniNorm.
	 	\end{align*}
%
for $u,u_1,u_2\in \MFshortTauEins{r}.$ We get a unique fixed point $\tilde{u}_2\in \MFshortTauEins{r}$  of $K_2^\nu$ and define
\begin{align*}
\tilde{\tau}_2:=\inf \left\{t\ge 0: \norm{u_2}_{L^q(0,t;L^q)}+\norm{u_2}_{L^{q_1}(0,t;L^{p_1})}\ge \nu\right\}\land r
\end{align*}
and $\tau_2:=\tau_1+\tilde{\tau}_2.$
Analogously to the proof of Proposition \ref{cutoffEquationLzwei},  one can show using \eqref{shiftStochIntegral}, that the pair $(u_2,\tau_2)$ with 
	\begin{equation*}
	u_{2}(t):= \left\{ \begin{aligned} 
	&u_1(t), \hspace{1.9cm}t\in [0,\tau_1],\\ &\tilde{u}_2(t-\tau_1), \hspace{1cm} t\in[\tau_1,\tau_2],
	\end{aligned}\right.
	\end{equation*}
defines a local mild solution of \eqref{ProblemStratonovich}. Iterating this procedure yields a sequence $\left(u_n,\tau_n\right)_{n\in\N}$ and with  $\tau_\infty:=\sup_{n\in\N}\tau_n$ and $\tau_0=0,$ we conclude that
\begin{align*}
	u(t):=\mathbf{1}_{\{t=0\}}u_0+\sum_{n=1}^\infty u_n(t) \mathbf{1}_{(\tau_{n-1},\tau_n]}(t) \quad \text{on  $\left\{t\le \tau_\infty\right\}$},
\end{align*}
the triple $\left(u,\left(\tau_n\right)_{n\in\N},\tau_\infty\right)$ is a local mild solution in the sense of Definition \ref{SolutionDef}.\\

\emph{Step 3.}
	 In order to show uniqueness, we take two local mild solutions $(u_1,\left(\sigma_{1,n}\right)_{n\in\N}\sigma_1)$ and $(u_2,\left(\sigma_{2,n}\right)_{n\in\N},\sigma_2)$ and set
	 \begin{align*}
	 	Z_{a,b}(u):=\norm{u}_{L^q(a,b;L^q)}+\norm{u}_{L^{q_1}(a,b;L^{p_1})}
	 \end{align*}
	 for $0\le a<b\le T.$ 
Moreover, we inductively define 
	 \begin{align*}
	 	\mu_1:=\inf \left\{t\in [0,\sigma_1): Z_{0,t}(u_1)\ge \nu\right\}\land
	 	\inf \left\{t\in [0,\sigma_2): Z_{0,t}(u_2)\ge \nu\right\}\land \sigma_1\land \sigma_2
	 \end{align*}
	 and
	 \begin{align*}
	 	\mu_{n+1}:=\inf \left\{t\in [\mu_n,\sigma_1): Z_{\mu_n,t}(u_1)\ge \nu\right\}\land
	 	\inf \left\{t\in [\mu_n,\sigma_2): Z_{\mu_n,t}(u_2)\ge \nu\right\}\land \sigma_1\land \sigma_2.
	 \end{align*}
	 	 The uniqueness from the first step and $\varphi_\nu(u_1,t)=1=\varphi_\nu(u_2,t)$ almost surely on $\left\{t\le \mu_1\right\}$
	 	  yield $u_1(t)=u_2(t)$ almost surely on $\{t< \sigma_1\land \sigma_2\}\cap \{t\le \mu_1\}.$ By an iteration procedure as above, this can be extended to  $\{t< \sigma_1\land \sigma_2\}\cap \{t\le \mu_n\}$ for all $n\in\N.$
	 
	 In order to show, that $\mu_n\to \sigma_1\land\sigma_2$ as $n\to \infty,$ it is sufficient that for all $m\in\N$ and almost all $\omega\in\Omega$ there is $n=n(\omega)$ with $\mu_{n(\omega)}(\omega)\ge \sigma_{1,m}(\omega)\land\sigma_{2,m}(\omega).$ Assume the opposite, i.e. there is $m\in\N$ with 
	 \begin{align*}
		 \Prob\left(\mu_n< \sigma_{1,m}\land\sigma_{2,m}\quad \forall n\in\N \right)>0.
	 \end{align*}
	 By the definition of $\mu_{n+1},$ we get on each interval $[\mu_n,\mu_{n+1}]$ either $Z_{\mu_n,\mu_{n+1}}(u_1)\ge \nu$ or\\ $Z_{\mu_n,\mu_{n+1}}(u_2)\ge \nu$ with positive probability. Without loss of generality, we assume that there is a subsequence $\left(\mu_{n_k}\right)_{k\in\N}$ with 
	 \begin{align*}
	 	\Prob\left(Z_{\mu_{n_k},\mu_{n_k+1}}(u_1)\ge \nu\quad \forall k\in\N\right)>0
	 \end{align*}
	 and therefore 
	 \begin{align*}
	 	\norm{u_1}_{L^q(0,\sigma_{1,m};L^q)}+\norm{u_1}_{L^{q_1}(0,\sigma_{1,m};L^{p_1})}\ge \sum_{k=1}^\infty \left(\norm{u_1}_{L^q(\mu_{n_k},\mu_{n_k+1};L^q)}+\norm{u}_{L^{q_1}(\mu_{n_k},\mu_{n_k+1};L^{p_1})}\right)=\infty
	 \end{align*}
	 which is infinite with positive probability.
	 This is a contradiction, since both $\norm{u_1}_{L^q(0,\sigma_{1,m};L^q)}$ and $\norm{u_1}_{L^{q_1}(0,\sigma_{1,m};L^{p_1})}$ are almost surely finite by $u_1\in \MFloc{\sigma_1}.$

\end{proof}


We close this section with remarks on possible slight generalizations of Theorem \ref{mainTheorem} a) 
and comment on the transfer of our method to the energy space $\Heins.$ 

\begin{Remark}
	In the proof of the local result, we did not use the special structure of the terms $B_1,$ $B_2$ and
	\begin{align}\label{muRemark}
		\mu_1:=-\frac{1}{2} \sumM \vert e_m\vert^2,\qquad \mu_2:=-\frac{1}{2} \sumM B_m^* B_m.
	\end{align}
	In fact, we only used $B_1,B_2 \in \mathcal{L}(\Lzwei, \HS(Y,\Lzwei)),$ $\mu_1\in \mathcal{L}(\Lzwei)\cap \mathcal{L}(L^{2\gamma}(\Rd))$ and $\mu_2\in \mathcal{L}(\Lzwei).$ But since \eqref{muRemark} is motivated by the Stratonovich noise and will be important for the global existence in the following section, we decided to start with the special case from the beginning.
	
	A generalization of the result from Theorem \ref{mainTheorem} from determistic initial values $u_0\in\Lzwei$ to  $u_0\in L^q(\Omega,\mathcal{F}_0;\Lzwei)$ is straightforward. By the standard localization technique (see e.g. \cite{UMDstochEvo}), a further generalization to $\mathcal{F}_0$-measurable $u_0:\Omega\to \Lzwei$ can be done by another localization procedure if one relaxes the condition $u\in \MFloc{\tau}$ from definition \ref{SolutionDef} to $u\in \mathbb{M}_\Filtration^0(\Omega,E_{[0,\tau)}),$ i.e. $u$ is a continuous $\Filtration$-adapted process in $\Lzwei$ and $\Filtration$-predictable in $L^{2\gamma}(\Rd)$ with 
	\begin{align*}
	\sup_{t\in [0,\tau]}\norm{u(t)}_{L^2}^p+ \norm{u}_{Y_\tau}^p<\infty \qquad \text{a.s.}
	\end{align*}
	For the sake of simplicity, we decided to restrict ourselves to deterministic initial values.
\end{Remark}


\begin{Remark}\label{HeinsTransfer}
	Barbu, R\"ockner and Zhang, \cite{BarbuH1}, and de Bouard and Debussche, \cite{BouardHeins}, also applied their strategy to construct solutions also in $\Heins.$  In contrast to the  $L^2$-case, the pathwise approach has a true advantage here, since it allows to adapt the deterministic fixed point argument in a ball of $L^\infty H^1\cap L^q W^{1,\alpha+1}$ equipped with the metric from $L^\infty L^2\cap L^q L^{\alpha+1}.$ Therefore, \cite{BarbuH1} obtain local wellposedness for all $H^1$-subcritical exponents $\alpha \in (1,1+\frac{4}{(d-2)_+})$ and a blow-up criterium involving the $H^1$-norm in the case of linear multiplicative noise. 
	
	Of course, it is also possible to deal with the $H^1$-problem with the method from the present paper, since the deterministic and stochastic Strichartz estimates are also true in $\Heins.$ 
	In this way, one can weaken 
	the regularity assumptions on the noise from \cite{BarbuH1} significantly to 
	\begin{align}\label{presentAssumptionHeins}
	\sumM \norm{e_m}_{W^{1,\infty}}^2<\infty
	\end{align}	
	and treat nonlinearities of degree
	\begin{align}\label{ExponentsHeins}
		\alpha \in (1,1+\frac{4}{d})\cup (2,1+\frac{4}{(d-2)_+}),\qquad \gamma \in [1,1+\frac{2}{d})\cup (2,1+\frac{2}{(d-2)_+}).
	\end{align}
	Hence, there is a gap compared to the natural $H^1$-subcritical exponents $\alpha$ for $d\ge 4$ and a gap in the range of noise exponents $\gamma$ for $d\ge 2.$ 
	The structure of the admissible sets for $\alpha$ and $\gamma$ indicate that one has to argue differently for small and large exponents. Indeed, the intervals $\alpha \in (1,1+\frac{4}{d})$ and $\gamma \in [1,1+\frac{2}{d})$ can be obtained by a fixed point argument in a ball in $L^\infty H^1\cap L^q W^{1,\alpha+1}$ equipped with the metric induced by the norm in $L^\infty L^2\cap L^q L^{\alpha+1},$ where the nonlinear terms are truncated as above. 
	
	 For twice continuous Fr\'echet differentiable nonlinear terms, however,  i.e. $\gamma>2$ and $\alpha>2,$ one can prove the contraction estimate in the full norm of $L^\infty H^1\cap L^q W^{1,\alpha+1},$ if the truncation also takes place in the stronger norms $L^q W^{1,\alpha+1}$ and $L^{\tilde{q}} W^{1,2\gamma}.$ For more details, we refer to \cite{BouardHeins}.
	
In view of the applications, the gap in \eqref{ExponentsHeins} is not too restrictive, since the main interest lies in the cubic case. This the advantage of the $H^1$-framework, since $\alpha=3$ is $H^1$-subcritical for $d=1,2,3,$ whereas in $L^2$ it is subcritical for $d=1$ and critical $d=2.$ 
\end{Remark}
\section{Global existence}

The goal of this section is to study global existence in the subcritical case $\alpha\in (1,1+\frac{4}{d})$ with conservative noise, i.e. $e_m$ is real valued for each $m\in\N.$ Let us recall that the local solution $\left(u,\left(\tau_n\right)_{n\in\N},\tau_\infty\right)$ is given by $u=u_n$ on $[0,\tau_n],$ where 
	\begin{align}\label{existenceTimes}
	\tau_n:=\inf \left\{t\in [0,T]: \norm{u_n}_{L^q(0,t;L^{\alpha+1})}+\norm{u_n}_{L^{\tilde{q}}(0,t;L^{2\gamma})}\ge n\right\}\land T,\qquad n\in\N,
	\end{align}
for exponents $q,\tilde{q}\in (2,\infty)$  satisfying the Strichartz conditions
\begin{align*}
	\frac{2}{q}+\frac{d}{\alpha+1}=\frac{d}{2},\qquad \frac{2}{\tilde{q}}+\frac{d}{2\gamma}=\frac{d}{2}.
\end{align*}
Moreover, $\tau_\infty=\sup_{n\in\N}\tau_n$ and $u_n$ is the solution of the truncated problem
\begin{equation}\label{ProblemStratonovichApproxGlobal}
\left\{
\begin{aligned}
\df u_n(t)&= \left(\im \Delta u_n(t)-\im \lambda\varphi_n(u_n,t) \vert u_n(t) \vert^{\alpha-1} u_n(t)-\frac{1}{2}\sumM e_m^2 \varphi_n(u_n,t)\vert u_n(t) \vert^{2(\gamma-1)} u_n(t)\right) \df t\\&\qquad-\im   \sumM e_m \varphi_n(u_n,t)  \vert u_n(t) \vert^{\gamma-1} u_n(t) \df \beta_m(t),\\
u(0)&=u_0.
\end{aligned}\right.
\end{equation}	
We remark that \eqref{ProblemStratonovichApproxGlobal} is a simplified version of \eqref{ProblemStratonovichApprox} to study nonlinear noise, i.e. $\gamma\neq 1.$ We will concentrate on this case, since the global existence result in the case $\gamma=1$ has already been treated by de Bouard and Debussche in \cite{BouardLzwei}, Proposition 4.1.
 
The strategy to prove global existence is determined by the definition of the existence times in \eqref{existenceTimes}: We need to find uniform bounds for $u_n$ in the space 
$L^q(0,T;L^{\alpha+1})\cap L^{\tilde{q}}(0,T;L^{2\gamma}).$ Note that this is a drawback of our approach based on the the truncation of the nonlinearities and can be avoided in the deterministic case, where the local existence result comes with a natural blow-up alternative in $\Lzwei$ and the mass conservation directly yields global existence.

 However, we overcome this problem by applying the deterministic and stochastic Strichartz estimates once again. The proof is inspired by the argument of de Bouard and Debussche for $\gamma=1$ mentioned above. Unfortunately, we cannot prove global existence for all $\gamma\in [1,1+\frac{2}{d})$ and have to restrict ourselves to 
	\begin{align*}
	1\le \gamma <\frac{\alpha-1}{\alpha+1} \frac{4+d(1-\alpha)}{4\alpha+d(1-\alpha)}+1.
	\end{align*}
The first ingredient is the mass conservation for the NLS in the stochastic setting. 

\begin{Prop}\label{massConservation}
	Let $\alpha \in (1,1+\frac{4}{d})$ and $\gamma\in [1,1+\frac{2}{d})$ and $e_m\in L^\infty(\Rd,\R)$ for each $m\in\N$ with $\sumM \norm{e_m}_{L^\infty}^2<\infty.$ Let $n\in\N$ and $u_n$ be the global mild solution of \eqref{ProblemStratonovichApprox} from Proposition \ref{cutoffEquationLzwei}.
	 Then, we have
	\begin{align}\label{massIdentity}
	\norm{u_n(t)}_\Lzwei\le \norm{u_0}_\Lzwei,\qquad t\in[0,T].
	\end{align} 
	almost surely. 
\end{Prop}

\begin{proof}	We concentrate on the case $\gamma>1.$ The linear case is simpler since one does not need the truncation of the noise and the correction term. 
	It is well known that the mild equation is equivalent to
			\begin{align}\label{weakEquation}
			u_n(t)=&  u_0+ \int_0^{t}  \left[\im \Delta u_n-\im\lambda \varphi(u_n)\vert u_n\vert^{\alpha-1} u_n-\frac{1}{2}\sumM e_m^2 \varphi(u_n)\vert u_n\vert^{2(\gamma-1)}u_n\right] \df s\nonumber\\
			&-\sumM \im \int_0^t e_m \varphi(u_n) \vert u_n\vert^{\gamma-1}u_n \df \beta_m
			\end{align}
	almost surely as an equation in $H^{-2}(\Rd).$
	We formally apply  Ito's formula to the It\^o process from \eqref{weakEquation} and the function
	$\mass: {\Lzwei} \to \R$ defined by $\mass(v):=\norm{v}_{L^2}^2,$ which is twice continuously Fr\'{e}chet-differentiable with 
	\begin{align*}
	\mass'[v]h_1&= 2 \Real \skpLzwei{v}{ h_1}, \qquad
	\mass''[v] \left[h_1,h_2\right]= 2 \Real \skpLzwei{ h_1}{h_2}
	\end{align*}
	for $v, h_1, h_2\in {\Lzwei}.$ This yields
			\begin{align}\label{Formal calculation}
			\norm{ u_n(t)}_{L^2}^2=&\norm{ u_0}_{L^2}^2+2 \int_0^t \Real \skpLzwei{ u_n}{\im  \Delta u_n-\im \lambda \varphi(u_n) \vert u_n\vert^{\alpha-1}u_n-\frac{1}{2}\sumM e_m^2 \varphi(u_n)\vert u_n\vert^{2(\gamma-1)}u_n} \df s\nonumber\\
						&- 2 \sumM \int_0^t \Real \skpLzwei{ u_n}{\im e_m \varphi(u_n) \vert u_n\vert^{\gamma-1}u_n }\df \beta_m
						+\sumM \int_0^t   
						\Vert  e_m \varphi(u_n) \vert u_n\vert^{\gamma-1}u_n\Vert_{L^2}^2\df s
			\end{align}
			almost surely  in $[0,T]$ and by the formal identities
						\begin{align*}
						\Real \skpLzwei{ u_n}{\im  \Delta u_n}&=0,\qquad 
						\Real \skpLzwei{ u_n}{\im \varphi(u_n)\vert u_n\vert^{\alpha-1}u_n}=0,\\
						\Real \skpLzwei{ u_n}{\im e_m \varphi(u_n)\vert u_n\vert^{\gamma-1}u_n}&=0,
						\end{align*} 
			and
			\begin{align*}
				\sumM \int_0^t \Vert  e_m \varphi(u_n) \vert u_n\vert^{\gamma-1}u_n\Vert_{L^2}^2\df s\le \sumM \int_0^t \Real \skpLzwei{u_n}{\sumM e_m^2 \varphi(u_n)\vert u_n \vert^{2(\gamma-1)}u_n}\df s,
			\end{align*}
			where we used that $\varphi(u_n)\in [0,1]$ implies  $\varphi(u_n)^2\le \varphi(u_n),$ we finally get 
			\begin{align*}
				\norm{ u_n(t)}_{L^2}^2\le \norm{ u_0}_{L^2}^2
			\end{align*}
			almost surely for all $t\in [0,T].$\\
			
	The calculation from above can be made rigorous by a regularization procedure via Yosida approximations $\Yosida:=\nu\left(\nu-\Delta\right)^{-1}$ for $\nu>0$ and a limit process $\nu\to \infty$ using the properties
	$\Yosida\in  \mathcal{L}(\Hs,H^{s+2}(\Rd))$ for $s\in\R$ and
	\begin{align}\label{YosidaProperties}
	&\Yosida f \to f \quad \text{in}\quad {E},\quad \nu \to \infty, \quad f\in{E}\nonumber\\
	&\hspace{1,5cm}\norm{\Yosida}_{{\mathcal{L}(E)}}\le 1,
	\end{align} 
	for $E=\Hs,$ $s\in\R,$ and $E=L^p(\Rd),$ $1<p<\infty.$
\end{proof}

Finally, we are ready to prove the global wellposedness result.

\begin{Prop}\label{globalExistence}
	Let $\alpha \in (1,1+\frac{4}{d})$ and $e_m\in L^\infty(\Rd,\R)$ for each $m\in\N$ with $\sumM \norm{e_m}_{L^\infty}^2<\infty.$ Let $(u_n)_{n\in\N}$ be the sequence of global mild solutions of \eqref{ProblemStratonovichApprox} from Proposition \ref{cutoffEquationLzwei}.
	Suppose that 
	\begin{align*}
	1\le \gamma <\frac{\alpha-1}{\alpha+1} \frac{4+d(1-\alpha)}{4\alpha+d(1-\alpha)}+1.
	\end{align*}	
	Then, we have 
	\begin{align*}
	\Prob\left(\bigcup_{n\in\N}\left\{\tau_n=T\right\}\right)=1.
	\end{align*}
	In particular, the pair $(u,\tau_\infty)$ is a unique global strong solution of \eqref{ProblemStratonovich}.
\end{Prop}

\begin{proof}
	\emph{Step 1.} As we mentioned above, we only consider $\gamma>1.$ The linear case is simpler and analogous to \cite{BouardLzwei}, Proposition 4.1. We want to prove that there is a uniform constant $C>0$ such that
	\begin{align}\label{uniformEstimateTruncatedSolution}
	\sup_{n\in\N} \E \norm{u_n}_{L^q(0,r;L^{\alpha+1})}\le C.
	\end{align}
	We set 
	$Y_r:=L^q(0,r;\LalphaPlusEins)$
	and fix $n\in\N$ as well as 
	\begin{align*}
	\delta:=1+\frac{d}{4}(1-\alpha), \qquad \tilde{\delta}:=1+\frac{d}{2}(1-\gamma),\qquad \theta=\frac{\alpha+1-2\gamma}{(\alpha-1)\gamma}.
	\end{align*}
	Then, we have $\frac{1}{2\gamma}=\frac{\theta}{2}+\frac{1-\theta}{\alpha+1}$ and 
	by interpolation and Proposition \ref{massConservation},
	\begin{align}\label{interpolationInequalityStrichartz}
	\norm{u_n}_{L^{\tilde{q}}(0,\sigma_n,L^{2\gamma})}\le \norm{u_n}_{L^\infty(0,\sigma_n;L^2)}^\theta \norm{u_n}_{L^q(0,\sigma_n;L^{\alpha+1})}^{1-\theta}
	\le \norm{u_0}_{L^2}^\theta \norm{u_n}_{L^q(0,\sigma_n;L^{\alpha+1})}^{1-\theta}
	\end{align}
	almost surely for all $t\in [0,T].$
	Let us recall that $u_n$ has the representation
	\begin{align*}
	u_n=U(\cdot)u_0+K_{det}^n u_n+K_{Strat}^n u_n+K_{Stoch}^n u_n\quad \text{in $M_\Filtration^p(\Omega,E_T).$}
	\end{align*}
	Next, we fix $\omega\in\Omega$ and $\sigma_n(\omega)\in (0,T]$ to be specified later. 
	Then, we apply the deterministic Strichartz inequalities from Proposition \ref{DetStrichartz} to estimate $K_{det}$ and $K_{Strat}$ (compare the proof of Proposition \ref{cutoffEquationLzwei}) and obtain
	\begin{align*}
	\norm{u_n}_{Y_{\sigma_n}}&\le C \norm{u_0}_{L^2}+ C \sigma_n^\delta \norm{u_n}_{Y_{\sigma_n}}^\alpha
	+C \sigma_n^{\tilde{\delta}} \norm{u_n}_{L^{\tilde{q}}(0,\sigma_n,L^{2\gamma})}^{2\gamma-1}\sumM \norm{e_m}_{L^\infty}^2+\norm{K_{Stoch}u_n}_{Y_{\sigma_n}}\nonumber\\
	&\le C \norm{u_0}_{L^2}+ C \sigma_n^\delta \norm{u_n}_{Y_{\sigma_n}}^\alpha
	+C \sigma_n^{\tilde{\delta}} \norm{u_0}_{L^2}^{(2\gamma-1)\theta} \norm{u_n}_{Y_{\sigma_n}}^{(2\gamma-1)(1-\theta)}\sumM \norm{e_m}_{L^\infty}^2+\norm{K_{Stoch}u_n}_{Y_{\sigma_n}}\nonumber\\
	&\le C \norm{u_0}_{L^2}+ C \sigma_n^{\delta} \left[1 + \sigma_n^{\tilde{\delta}-\delta}\norm{u_0}_{L^2}^{(2\gamma-1)\theta}\sumM \norm{e_m}_{L^\infty}^2\right]\norm{u_n}_{Y_{\sigma_n}}^\alpha\\&
	\quad+C \sigma_n^{\tilde{\delta}} \norm{u_0}_{L^2}^{(2\gamma-1)\theta} \sumM \norm{e_m}_{L^\infty}^2+\norm{K_{Stoch}u_n}_{Y_{\sigma_n}}\nonumber\\
	&\le K_n+ C_1 \sigma_n^{\delta} \norm{u_n}_{Y_{\sigma_n}}^\alpha
	\end{align*}
	with
	\begin{align*}
	K_n:=&C \norm{u_0}_{L^2}+C T^{\tilde{\delta}} \norm{u_0}_{L^2}^{(2\gamma-1)\theta} \sumM \norm{e_m}_{L^\infty}^2+\norm{K_{Stoch}u_n}_{Y_T},\\	
	C_1:=&C \left[1 + T^{\tilde{\delta}-\delta} \norm{u_0}_{L^2}^{(2\gamma-1)\theta}\sumM \norm{e_m}_{L^\infty}^2\right].
	\end{align*}
	W.l.o.g we assume $u_0\neq 0$ and thus $K_n>0.$ We conclude
	\begin{align*}
	\frac{\norm{u_n}_{Y_{\sigma_n}}}{K_n}\le 1+  C_1 \sigma_n^{\delta} K_n^{\alpha-1}\left(\frac{\norm{u_n}_{Y_{\sigma_n}}}{K_n}\right)^\alpha.
	\end{align*}
	Now, the following fact 
	\begin{align}\label{calculusFact}
	\forall x\ge 0\exists c_1\le 2, c_2>c_1: \quad x\le 1+\frac{x^\alpha}{2^{\alpha+1}}\quad \Rightarrow \quad x\le c_1 \quad \text{or}\quad x\ge c_2.
	\end{align}
	from elementary calculus yields
	\begin{align*}
	\norm{u_n}_{Y_{\sigma_n}}\le c_1 K_n \le 2 K_n,
	\end{align*}
	if we choose $\sigma_n$ according to 
	$C_1 \sigma_n^{\delta} K_n^{\alpha-1}\le\frac{1}{2^{\alpha+1}},$
	which is fulfilled by
	\begin{align*}
	\sigma_n=C_1^{-\frac{1}{\delta}} \left(2^{\alpha+1}K_n^{\alpha-1}\right)^{-\frac{1}{\delta}}\land T.
	\end{align*}
	Note that the second alternative in \eqref{calculusFact} can be excluded because of $\norm{u_n}_{Y_0}=0$ and the continuity of the map $t\mapsto \norm{u_n}_{Y_t}.$ 	
	Next, we decompose $\Omega=\Omega_1\cup\Omega_2$ with
	\begin{align*}
	\Omega_1:= \left\{C_1^{-\frac{1}{\delta}} \left(2^{\alpha+1}K_n^{\alpha-1}\right)^{-\frac{1}{\delta}}<T\right\},\qquad \Omega_2:= \left\{C_1^{-\frac{1}{\delta}} \left(2^{\alpha+1}K_n^{\alpha-1}\right)^{-\frac{1}{\delta}}\ge T\right\}.
	\end{align*}
	Fix $\omega \in \Omega_1$ and define 
	$N:=\lfloor\frac{T}{\sigma_n}\rfloor.$ 
	Using the abbreviations $Y_N:=L^q(N\sigma_n,T;\LalphaPlusEins)$ and
	\begin{align*}
	Y_j:=L^q(j\sigma_n,(j+1)\sigma_n;\LalphaPlusEins),\qquad j=0,\dots, N-1,
	\end{align*}
	analogous estimates as above yield
	\begin{align*}
	\norm{u_n}_{Y_j}&\le C \norm{u_n(j\sigma_n)}_{L^2}+ C \sigma_n^\delta \norm{u_n}_{Y_j}^\alpha
	+C \sigma_n^{\delta} \norm{u_n}_{L^{\tilde{q}}(j\sigma_n,\sigma_{n+1},L^{2\gamma})}^{2\gamma-1}\sumM \norm{e_m}_{L^\infty}^2+\norm{K_{Stoch}u_n}_{Y_j}
	\end{align*}
	for all $j=0,\dots, N$ and hence, we obtain $\norm{u_n}_{Y_j}\le 2 K_n.$	
	We conclude
	\begin{align}\label{Omega1Estimate}
	\norm{u_n}_{Y_T}\le \sum_{j=0}^N \norm{u_n}_{Y_j}\le 2 \left(N+1\right) K_n\le 2 \left(\frac{T}{\sigma_n}+1\right) K_n
	\le 2 K_n+ 2^{\frac{\alpha+1}{\delta}+1}C_1^\frac{1}{\delta} T K_n^{\frac{\alpha-1}{\delta}+1}.
	\end{align} 
	Due to $\norm{u_n}_{Y_T}\le 2 K_n$ on $\Omega_2,$ the estimate \eqref{Omega1Estimate} holds almost surely. 
	We set $p:=\frac{\alpha-1}{\delta}+1$ and integrate over $\Omega$ to obtain
	\begin{align*}
	\norm{u_n}_{L^1(\Omega,Y_T)}
	&\le 2 \E \Big[K_n\Big]+ 2^{\frac{\alpha+1}{\delta}+1}C_1^\frac{1}{\delta} T \E \Big[K_n^{\frac{\alpha-1}{\delta}+1}\Big]\lesssim 1+ \E \norm{K_{Stoch}u_n}_{Y_T}+   \E \norm{K_{Stoch}u_n}_{Y_T}^{p}\\
	&\le 1+\norm{K_{Stoch}u_n}_{L^p(\Omega,Y_T)}+\norm{K_{Stoch}u_n}_{L^p(\Omega,Y_T)}^p
	\end{align*} 
	Now, we choose $\tilde{p}\in (p\gamma,\infty)$ according to 
	$\frac{1}{p\gamma}=\frac{\theta}{\tilde{p}}+\frac{1-\theta}{1}$ and by the estimate
	\begin{align*}
	\norm{K_{Stoch}u_n}_{L^p(\Omega,Y_T)}&\lesssim \norm{\varphi_n(u,s) \vert u_n\vert^{\gamma-1}u_n}_{L^p(\Omega,L^2(0,T;L^2))}
	\le T^{\delta} \norm{u_n}_{L^{p\gamma}(\Omega,L^{\tilde{q}}(0,T;L^{2\gamma}))}^\gamma\\
	&\le T^{\delta} \norm{u_n}_{L^{\tilde{p}}(\Omega,L^{\infty}(0,T;L^{2}))}^{\gamma\theta}\norm{u_n}_{L^{1}(\Omega;Y_T)}^{\gamma(1-\theta)}
	=T^{\delta} \norm{u_0}_{L^2}^{\gamma\theta}\norm{u_n}_{L^{1}(\Omega;Y_T)}^{\gamma(1-\theta)}\\
	&\lesssim \norm{u_n}_{L^{1}(\Omega;Y_T)}^{\gamma(1-\theta)},
	\end{align*}	
	we end up with
	\begin{align*}
	\norm{u_n}_{L^1(\Omega,Y_T)}
	&
	\lesssim 1+ \norm{u_n}_{L^{1}(\Omega;Y_T)}^{\gamma(1-\theta)}+\norm{u_n}_{L^{1}(\Omega;Y_T)}^{p\gamma(1-\theta)}
	\lesssim 1+\norm{u_n}_{L^{1}(\Omega;Y_T)}^{p\gamma(1-\theta)}.
	\end{align*} 
	In particular, there is $C=C(\norm{u_0}_{L^2},\norm{e_m}_{\ell^2(\N,L^\infty)},T,\alpha,\gamma)>0$ with 
	\begin{align*}
	\sup_{n\in\N}\norm{u_n}_{L^1(\Omega,Y_T)}\le C,\qquad n\in\N,
	\end{align*}
	if we have
	\begin{align*}
	p\gamma(1-\theta)<1\qquad \Leftrightarrow\qquad \gamma <\frac{\alpha-1}{\alpha+1} \frac{4+d(1-\alpha)}{4\alpha+d(1-\alpha)}+1.
	\end{align*}
	\emph{Step 2.} 
	Using the result of the first step and taking the expectation in \eqref{interpolationInequalityStrichartz}, we obtain
	\begin{align*}
	\norm{u_n}_{L^1(\Omega, L^{\tilde{q}}(0,T;L^{2\gamma}))}\le \norm{u_0}_{L^2}^\theta C^{1-\theta}
	\end{align*}
	and the definition of $\tau_n$ followed by the Tschebycheff inequality and \eqref{uniformEstimateTruncatedSolution} yield
	\begin{align*}
	\Prob\left(\tau_n=T\right)&=\Prob\left(\norm{u_n}_{Y_T}+\norm{u_n}_{L^{\tilde{q}}(0,T;L^{2\gamma})}\le n\right)\ge 1-\frac{\norm{u_n}_{L^1(\Omega,Y_T)}+\norm{u_n}_{L^1(\Omega,L^{\tilde{q}}(0,T;L^{2\gamma}))}}{n}\\
	&\ge 1-\frac{C+\norm{u_0}_{L^2}^\theta C^{1-\theta}}{n}.
	\end{align*}
	By the continuity of the measure, we conclude 
	\begin{align*}
	\Prob\left(\tau_\infty=T\right)\ge\Prob\left(\bigcup_{n\in\N}\left\{\tau_n=T\right\}\right)=\lim_{n\to \infty} \Prob\left(\tau_n=T\right)=1.
	\end{align*}
\end{proof}

\begin{Remark}
	We comment on the cases, which have been excluded for the global existence result in Proposition \ref{globalExistence}. 
	The proof cannot be applied to the critical case $\alpha=1+\frac{4}{d},$ where we have $\delta=0,$ since the strategy crucially relies on $\delta>0$ to apply \eqref{calculusFact}. But global existence for general $L^2$-initial data cannot be expected in this case, anyway, since there are blow-up examples in the deterministic setting for the focusing nonlinearity, see \cite{Merle}.\\
	
	The restriction to real valued coefficients $e_m$ can be dropped in the case $\gamma=1.$ Indeed, one can proceed as in \cite{BarbuL2} to deduce the estimate  
	\begin{align}\label{L1BoundednessOfMass}
			\E \Big[\sup_{t\in[0,T]}\norm{u_n(t)}_{L^2}^p\Big]\le D_p,\qquad p\in[1,\infty),
	\end{align}
	based on the formula
		\begin{align*}
		\norm{u_n(t)}_\Lzwei^2=\norm{u_0}_\Lzwei^2- 2 \int_0^t \Real \skpLzwei{ u_n(s)}{\im B u_n(s) \df W(s)},\qquad t\in[0,T].
		\end{align*} 
	and a Gronwall argument which is restricted to $\gamma=1.$ Then, the estimate \eqref{L1BoundednessOfMass} can be used to substitute the pathwise estimate from Proposition \ref{massConservation} in the proof of Proposition \ref{globalExistence}.
\end{Remark}

\textbf{Acknowledgement:} The author gratefully acknowledges financial  support by the Deutsche Forschungsgemeinschaft (DFG) through CRC 1173.

\bibliographystyle{alpha}
\bibliography{ReferencesAll}
\Addresses
\end{document}